\documentclass[12pt]{amsart}
\usepackage{mathrsfs}
\usepackage{amsfonts}
\usepackage[centertags]{amsmath}
\usepackage{amssymb}
\usepackage{amsthm}
\usepackage{graphicx}
\usepackage{float}
\usepackage[all]{xy}
\usepackage{tikz-cd}
\usepackage{caption}
\usepackage[colorlinks]{hyperref}
\hypersetup{
bookmarksnumbered,
pdfstartview={FitH},
breaklinks=true,
linkcolor=blue,
urlcolor=blue,
citecolor=blue,
bookmarksdepth=2
}

\usepackage{enumerate}
\usepackage[a4paper, portrait, margin=0.9in]{geometry}
\usepackage{tikz-cd}
\usepackage{rotating}
\usetikzlibrary{matrix,arrows,decorations.pathmorphing}
\usepackage{listing}
\usepackage{color}

\theoremstyle{plain}
   
   \newtheorem{theorem}{Theorem}[section]
   \newtheorem{proposition}[theorem]{Proposition}
   
   \newtheorem{lemma}[theorem]{Lemma}
   
   \newtheorem{conjecture}[theorem]{Conjecture}
   
\theoremstyle{definition}
   \newtheorem{definition}[theorem]{Definition}
   
   \newtheorem{example}[theorem]{Example}

   \newtheorem{remark}[theorem]{Remark}

\numberwithin{equation}{section}

\newcommand{\CC}{{\mathbb {C}}}

\newcommand{\ZZ}{{\mathbb {Z}}}

\newcommand{\xx}{{\mathbf{x}}}

\DeclareMathOperator{\Mat}{Mat}

\DeclareMathOperator{\Gr}{Gr}

\newcommand\scalemath[2]{\scalebox{#1}{\mbox{\ensuremath{\displaystyle #2}}}}

\def\t{\widetilde}
\def\T{\mathbb{T}}
\def\A{\mathcal{A}}
\def\XX{\mathcal{X}}

\def\x{{\bf{x}}}

\def\p{{\bf{p}}}
\def\P{\mathbb{P}}
\def\S{\Sigma}

\begin{document}

\title{Quasi-homomorphisms of quantum cluster algebras}
\author{Wen Chang, Min Huang, Jian-Rong Li}

\address{Wen Chang, School of Mathematics and Statistics, Shaanxi Normal University, Xi'an, China}
\email{changwen161@163.com}
\address{Min Huang, School of Mathematics (Zhuhai), Sun Yat-sen University, Zhuhai, China}
\email{huangm97@mail.sysu.edu.cn}
\address{Jian-Rong Li, Faculty of Mathematics, University of Vienna, Oskar-Morgenstern-Platz 1, 1090 Vienna, Austria}
\email{lijr07@gmail.com}

\date{}

\maketitle

\begin{abstract}
In this paper, we study quasi-homomorphisms of quantum cluster algebras, which are quantum analogy of quasi-homomorphisms of cluster algebras introduced by Fraser.  

For a quantum Grassmannian cluster algebra $\CC_q[\Gr(k,n)]$, we show that there is an associated braid group and each generator $\sigma_i$ of the braid group preserves the quasi-commutative relations of quantum Pl\"{u}cker coordinates and exchange relations of the quantum Grassmannian cluster algebra. We conjecture that $\sigma_i$ also preserves $r$-term ($r \ge 4$) quantum Pl\"{u}cker relations of $\CC_q[\Gr(k,n)]$ and other relations which cannot be derived from quantum quantum Pl\"{u}cker relations (if any). Up to this conjecture, we show that $\sigma_i$ is a quasi-automorphism of $\CC_q[\Gr(k,n)]$ and the braid group acts on $\CC_q[\Gr(k,n)]$.
\end{abstract}

\tableofcontents\addtocontents{toc}{\setcounter{tocdepth}{1}}

\section{Introduction}
Cluster algebras were introduced by Fomin and Zelevinsky in \cite{FZ02}.
To study the symmetries of the cluster algebras, Assem, Schiffler, and Shramchenko introduced cluster automorphisms in \cite{ASS}.
The concept was generalized to quasi-homomorphisms of cluster algebras by Fraser in \cite{Fra16}. More relations between cluster automorphisms and quasi-automorphisms were discovered in \cite{CS19}. Both the automorphisms and quasi-homomorphisms of cluster algebras are essential tools in the study of symmetries of cluster algebras. 

It is natural to consider quantum version of ``quasi-homomorphisms of cluster algebras''. In \cite{KQW22}, Kimura, Qin, and Wei introduced twist automorphisms for upper cluster algebras and cluster Poisson algebras with coefficients. In Section 8 of \cite{KQW22}, they introduced quantization of twist automorphisms for quantum upper cluster algebras.  

The goal of this paper is to study quasi-homomorphisms of quantum cluster algebras and to study braid group action on the quantum Grassmannian cluster algebra $\CC_q[\Gr(k,n)]$. 

We write a definition of a quasi-homomorphism $f$ from a quantum cluster algebra $\mathcal{A}_q$ to another quantum cluster algebra $\mathcal{A}_q'$, see Definition \ref{def:QQH} and Remark \ref{remark:literature}, which is more similar to the non-quantum version of quasi-homomorphisms given by Fraser \cite{Fra16}. We require that $f: \mathcal{A}_q \to \mathcal{A}_q'$ is an algebra homomorphism, sends quantum clusters in $\mathcal{A}_q$ to quantum clusters in $\mathcal{A}_q'$ (up to frozen variables), preserves the mutable part of exchange matrix, and sends quantum cluster $\mathcal{X}$-variables (denoted by $\widehat{y}_i$) to quantum cluster $\mathcal{X}$-variables. The quantum quasi-homomorphism in Definition \ref{def:QQH} is slightly different from the quantum twist automorphism in \cite{KQW22}. The quantum quasi-homomorphism in Definition \ref{def:QQH} is defined for quantum cluster algebras, while quantum twist automorphism in \cite{KQW22} is defined for quantum upper cluster algebras. We prove some properties of quantum quasi-homomorphisms which will be used in the study of braid group actions on quantum Grassmannian cluster algebras, see Proposition \ref{prop:equ-quasi-homo}. 

For $k \le n$, the Grassmannian $\Gr(k,n)$ is the set of $k$-dimensional subspaces in an $n$-dimensional vector space. Denote by $\CC[\Gr(k,n)]$ its coordinate ring. Scott proved that there is a cluster algebra structure on $\CC[\Gr(k,n)]$ \cite{Sco06}. The coordinate ring $\CC[\Gr(k,n)]$ is called a \emph{Grassmannian cluster algebra}. Fraser \cite{Fra17} has studied braid group action on $\CC[\Gr(k,n)]$. He introduced regular automorphisms $\sigma_i, 1 \leq i \leq d-1,$ on $\Gr(k,n)$, which induce quasi-automorphisms $\sigma^*_i$ (for simplicity, we also write $\sigma_i^*$ as $\sigma_i$) on the coordinate ring $\CC[\Gr(k,n)]$, where $d={\rm gcd}(k,n)$, and he proved that these automorphisms give a braid group action on $\CC[\Gr(k,n)]$.

In this paper we take $\CC_q = \CC(q^{\pm 1/2})$. The quantized coordinate ring $\CC_q[\Gr(k,n)]$ is the subalgebra of the quantum matrix algebra $\CC_q[M(k,n)]$ generated by the quantum Pl\"{u}cker coordinates, see \cite{TT91, GL14, Lau06, Lau10} and Section \ref{subsec:quantum Grassmannian coordinate ring}. Grabowski and Launois \cite{GL11, GL14} proved that the quantum deformation $\CC_q[\Gr(k,n)]$ of $\CC[\Gr(k,n)]$ has a quantum algebra structure. 

Jensen, King and Su showed (see the paragraphs before Theorem 1.3 in \cite{JKS22}) that there is a quantum cluster algebra containing cluster variables $X_J$ for all $k$-subsets $J$ of $[n]$, whose quasi-commutation rules are the same as the quasi-commutation rules of the quantum minors $\Delta_q^J$ of $\CC_q[\Gr(k,n)]$. They showed that the quantum cluster algebra is isomorphic to the quantum coordinate ring $\CC_q[\Gr(k,n)]$ by identifying $X_J$ with $\Delta_q^J$ for all $J$, see Theorem 1.3 in \cite{JKS22} (note that we take $\CC_q = \CC(q^{\pm 1/2})$). 

There is a bijection $\varphi$ between cluster variables in $\CC[\Gr(k,n)]$ and cluster variables in $\CC_q[\Gr(k,n)]$, see Section \ref{subsec:the bijection varphi between cluster monomials and quantum cluster monomials}. It is shown in \cite{Fra17} that in $\CC[\Gr(k,n)]$, for every cluster variable $x$, $\sigma_i(x) = y_1 y_2 \cdots y_r \tilde{x}$, where $r$ is a postive integer, $y_i$ is a frozen variable or the inverse of a frozen variable, and $\tilde{x}$ is a cluster variable. In $\CC_q[\Gr(k,n)]$, we define $\sigma_i(\varphi(x))$ to be the quantum cluster Laurent monomial $[\varphi(y_1)\varphi(y_2)\cdots \varphi(y_r) \varphi(\tilde{x})]$, see Sections \ref{subsec:commutative quantum cluster monomials}, \ref{subsec:the maps sigmai and main results}, and \ref{subsec:commutative quantum cluster monomials} for the definition of the notation $[x_1^{k_1} \cdots x_m^{k_m}]$. 

We prove that every map $\sigma_i$ extends to a map (also denoted by $\sigma_i$) that preserves the quasi-commutation relations of quantum Pl\"{u}cker coordinates and exchange relations of the quantum Grassmannian cluster algebra $\CC_q[\Gr(k,n)]$. We conjecture that $\sigma_i$ also preserves $r$-term ($r \ge 4$) quantum Pl\"{u}cker relations of $\CC_q[\Gr(k,n)]$ and other relations which cannot be derived from quantum Pl\"{u}cker relations (if any), see Conjecture \ref{conj:sigmai preserves r term plucker relations r is 4 or more}. Up to this conjecture, we show that $\sigma_i$ is a quasi-automorphism of $\CC_q[\Gr(k,n)]$ and the braid group acts on $\CC_q[\Gr(k,n)]$, see Theorem \ref{thm: braids are QQH}. This braid group action gives a way to discover many quantum cluster variables and quantum clusters of the quantum Grassmannian cluster algerba.

Symmetries of quantum Grassmannians have been studied in \cite{LL11}, \cite{AG12}, \cite{LL23}. Launois and Lenagan proved that a certain cycling map $\rho$ defines a  $\CC_q$-algebra isomorphism from $\CC_q[\Gr(k,n)]$ to its twisted algebra $T(\CC_q[\Gr(k,n)])$ and they showed that the cycling map $\rho$ is not a $\CC_q$-algebra automorphism of $\CC_q[\Gr(k,n)]$. For example, in $\CC_q[\Gr(2,4)]$, the map $\sigma_1$ is similar to but different from the cycling map $\rho$. Unlike $\rho$, the map $\sigma_1$ is a $\CC_q$-algebra automorphism on $\CC_q[\Gr(2,4)]$ (see Section \ref{sec:Gr24}). The braid group action gives a large family of $\CC_q$-algebra automorphisms on the quantum cluster algebra $\CC_q[\Gr(k,n)]$, see Lemma \ref{lem:sigmai and sigmai inverse are algebra automorphisms}. 

Marsh and Scott \cite{MS16} defined a twist automorphism on $\CC[\Gr(k,n)]$ which sends cluster variables to cluster variables up to multiplying frozen variables. This twist automorphism is closed related to Berenstein–Fomin–Zelevinsky twist automorphisms on unipotent cells \cite{BFZ96}. We expect that there is a quantum version of Marsh-Scott's twist map on $\CC_q[\Gr(k,n)]$ which should be closely related to the quantum twist map studied in \cite{KO19}.

The paper is organized as follows. In Section \ref{sec:quasi-hom of quantum cluster algebras}, we introduce the concept of quasi-homomorphisms of quantum cluster algebras. In Section \ref{sec:braid group action on quantum Grassmannian cluster algebras}, we study braid group actions on quantum Grassmannian cluster algebras. In Section \ref{sec:prove main theorem}, we prove our main result: Theorem \ref{thm: braids are QQH}. 

\subsection*{Acknowledgements}

We would like to thank Chris Fraser, Jan E. Grabowski, and Bach Nguyen for very helpful discussions. We would like to thank Fan Qin for telling us their work of twist automorphisms for upper cluster algebras and cluster Poisson algebras with coefficients, and quantum twist automorphisms for quantum upper cluster algebras in \cite{KQW22}, and for very helpful discussions. The authors would like to thank the anonymous referees for their very helpful comments and suggestions which have helped us to improve considerably the paper.

The first author is supported by Shaanxi Normal University and NSF of China: 12271321, and the project of Youth Innovation Team of Universities of Shandong Province (No. 2022KJ314). The second author is supported by NSF of China: 12101617, Guangdong Basic and Applied Basic Research Foundation 2021A1515012035, and Fundamental Research Funds for the Central Universities, Sun Yat-sen University (23qnpy51). The third author is supported by the Austrian Science Fund (FWF): P-34602. 

\subsection*{Notation}

Throughout the paper, we use $\ZZ$ as the set of integers and use the notation $[b]_+=\max(b,0)$ for $b\in \ZZ$. We use $\Mat_{m\times n}(\ZZ)$ to denote the set of $m\times n$ matrices over $\ZZ$. For a matrix $M=(b_{ij})\in \Mat_{m\times n}(\ZZ)$, we also use the notation $[M]_+=([b_{ij}]_+)$. For any positive integer $m$, we denote by $[1,m]$ the (ordered) set of positive integers smaller than $m+1$.
We will use $|I|$ to denote the cardinality of a set $I$.  

\section{Quasi-homomorphisms of quantum cluster algebras} \label{sec:quasi-hom of quantum cluster algebras}

We will introduce the notion of quasi-homomorphisms of quantum cluster algebras in this section.
We start with briefly recalling some background which will be used in the following, see details in \cite{BZ05}.

\subsection{Background on quantum cluster algebras}

\subsubsection{Quantum seeds}
Let $m$ and $n$ be two positive integers. A quantum seed is a triple $\Sigma=(\x,\t{B},\Lambda)$, where
\begin{itemize}
\item  $\x=(x_{1}, \ldots  ,x_{m})$ is an ordered set with $m$ elements;

\item  $\t{B}=\left(\begin{smallmatrix}B\\ C\end{smallmatrix}\right)\in \Mat_{m\times n}(\ZZ)$ is an extended skew-symmetrizable matrix, that is, there exists a diagonal matrix $D$ with positive integer entries such that $DB$ is skew-symmetric, where $B$ is the $n\times n$ submatrix of $\t{B}$ consisting of the first $n$ rows;

\item $\Lambda\in \Mat_{m\times m}(\ZZ)$ is a skew-symmetric matrix such that $(\t{B},\Lambda)$ is a compatible pair, that is ${\t{B}}^t\Lambda=(D, 0_{n\times (m-n)})$.
\end{itemize}
The set $\x$ is the \emph{cluster} of the quantum seed $\Sigma$. The elements in $\p=\{x_{n+1},\ldots,x_m\} \subset \x$ are called the \emph{coefficient variables} and the elements in $\x\setminus \p$ are called the \emph{cluster variables} of $\S$. The $n\times n$ matrix $B$ is called the \emph{exchange matrix} of $\S$; its rows are called \emph{exchangeable rows}, and the remaining rows of $\t{B}$ are called \emph{frozen rows}. We always assume that both $\t{B}$ and $B$ are indecomposable matrices, and we also assume that $n>1$ for convenience.

\subsubsection{Quantum torus} \label{subsec:quantum torus}
Denote $\CC_q = \CC(q^{\pm \frac{1}{2}})$, where $q$ is an indeterminate. A quantum torus $\T$ of rank $m \in \ZZ_{>0}$ is a $\CC_q$-algebra which has basis $\{X^a \mid a\in \ZZ^{m}\}$ with multiplication:
\begin{equation}\label{eq:qtorus1}
X^aX^b=q^{\frac{1}{2}\Lambda(a,b)}X^{a+b},
\end{equation}
where $\Lambda(-,-)$ is the bilinear form given by a skew-symmetric matrix $\Lambda$.

Let $\{e_1,\ldots,e_m\}$ be the standard basis of $\ZZ^m$, then we have
\begin{equation}\label{eq:qtorus2}
X^{e_i}X^{e_j}=q^{\frac{1}{2}\Lambda(e_i,e_j)}X^{e_i+e_j},X^{e_i}X^{e_j}
=q^{\Lambda(e_i,e_j)}X^{e_j}X^{e_i}.
\end{equation}

We associate with $\T$ the $\CC$-linear \emph{bar-involution} by setting
$$\overline{q^{\frac{1}{2}}X^a}=q^{-\frac{1}{2}}X^a$$
on each basis element. 
An element $X$ in $\T$ is a \emph{bar-invariant} if it is maintained by the bar-involution, that is, $\overline X=X$.

Given a quantum seed $\Sigma=(\x,\t{B},\Lambda)$ with notations as above, we associate to it a quantum torus $\T(\Sigma)$, where $X^{e_i}=x_i$ and $\Lambda(-,-)$ is the bilinear form given by $\Lambda$.
We denote by $\P$ the monoid (the multiplication is inherited from $\T(\Sigma)$) consisting of elements $X^{(0,a)}$ with $0\in \ZZ^n$ and $a\in \ZZ^{m-n}$.

\subsubsection{Quantum mutations}

Denote by $b_{ij}$ the entry in $\t{B}$ located at position $(i,j)$ for $1 \le i \le m$, $1 \le j \le n$ and by $b_k$ the $k$-th column of $\t{B}$ for $1\leq k \leq n$. Similarly we use notations $\lambda_{ij}$ and $\lambda_k$ for $\Lambda$ with $1\leq i,j,k \leq m$.
The quantum seed $\mu_k(\S)=(\mu_k(\x),\mu_k(\t{B}),\mu_k(\Lambda))$ obtained by the \emph{mutation} of $\S$ in direction $k$ is given by:
\begin{itemize}
				\item $\mu_k(\x) = (\x \setminus \{x_k\}) \sqcup \{\mu_k(x_k)\}$ where $\mu_k(x_k)$ is an element in $\T(\Sigma)$ defined by
\begin{equation}
\label{eq:cluster-variable-mutation}
\mu_k(x_k) = X^{-e_k+[b_k]_+}+X^{-e_k+[-b_k]_+}.
\end{equation}
				\item $\mu_k(\widetilde B)=(b'_{ji})_{m\times n} \in M_{m\times n}(\ZZ)$ is given by
\begin{equation}
\label{eq:matrix-mutation}
b'_{ji} = \left\{\begin{array}{ll}
						- b_{ji} & \textrm{ if } i=k \textrm{ or } j=k~; \\
						b_{ji} + [-b_{jk}]_+b_{ki} + b_{jk}[b_{ki}]_+ & \textrm{ otherwise.}
					\end{array}\right.
\end{equation}
               \item $\mu_k(\Lambda)=(\lambda'_{ji})_{m\times m} \in M_{m\times m}(\ZZ)$ is given by
\begin{equation}
\label{eq:lambda-matrix-mutation}
\lambda'_{ji} = \left\{\begin{array}{ll}
					    \lambda_{jk}+ \lambda^t_j [b_k]_+	& \textrm{ if } i=k~\vspace{1mm}; \\
					    \lambda_{ki}- \lambda^t_i [b_k]_+	& \textrm{ if } j=k~; \\
						\lambda_{ji} & \textrm{ otherwise.}
					\end{array}\right.
\end{equation}
\end{itemize}
It can be check that $(\mu_k(\t{B}),\mu_k(\Lambda))$ is still a compatible pair and $\mu(\Sigma)$ is a quantum seed.

\subsubsection{Quantum cluster algebra}
Let $\XX$ be the set of cluster variables in the quantum seeds obtained by iterated mutations from a quantum seed $\Sigma$. Then the algebra
\[\A_q=\A_q(\Sigma):=\CC_q\mathbb P[\XX]
\]
algebraically generated by elements in $\XX$ over $\CC_q\mathbb P$ is called the {\it quantum cluster algebra} of {\it rank} $n$ defined by $\Sigma$, which is a $\CC_q\mathbb P$-subalgebra of the quantum torus $\T(\Sigma)$. The elements in $\XX$ are called \emph{cluster variables}, and the elements in $\CC_q\mathbb P$ are called \emph{coefficients} of $\A_q$.

\subsection{Quasi-homomorphisms of quantum cluster algebras}

Let $\A_q$ and ${\A}'_q$ be two quantum cluster algebras of rank $n$.
Assume that $\S=\{\x=\{x_1,\ldots,x_m\},\t{B}=\left(\begin{smallmatrix}B\\ C\end{smallmatrix}\right),\Lambda\}$ and $\S'=\{\x'=\{x'_1,\ldots,x'_{m'}\},\t{B'}=\left(\begin{smallmatrix}B'\\ C'\end{smallmatrix}\right),\Lambda'\}$ be the initial seeds of them respectively. Correspondingly, we have quantum torus $\T$ and $\T'$.

\begin{definition}\label{def:QQH}
We call a $\CC_q$-algebra homomorphism $f: \A_q \to \A'_q$ a quantum quasi-homomorphism (with respect to $\S$ and $\S'$) if the following conditions are satisfied:
\begin{enumerate}
\item $B=B'$;

\item $f(x_i)=X'^{e'_i+\widetilde h_i}$ for all $1\leq i \leq n$, and $f(x_{j})=X'^{\t{l}_j}$ for all $n+1\leq i \leq m$, where $\t{h}_i=\left(\begin{smallmatrix}0\\ h_i\end{smallmatrix}\right)$ and $\t{l}_j=\left(\begin{smallmatrix}0\\ l_j\end{smallmatrix}\right)$ with $0\in \ZZ^{n}, h_i\in \ZZ^{{m'}-n}$ and $l_j\in \ZZ^{{m'}-n}$ such that $$f(\hat{y}_i)=\hat{y}'_i$$ for all $1\leq i \leq n$, where $\hat{y}_i=X^{b_i}$ and $\hat{y}'_i=X'^{b'_i}$.
\end{enumerate}
\end{definition}

\begin{remark} \label{remark:literature}
In Definition 3.1.3 of \cite{KQW22}, the variation map between Laurent polynomial rings preserving the Poisson structure was defined. It commutes with mutation sequences by Proposition 3.3.2 in \cite{KQW22}. It restricts to a map between upper cluster algebras by Lemma 3.3.4 in \cite{KQW22}. Since the Poisson structure is preserved, the map has a natural quantization (see the first paragraph in Section 8.2 in \cite{KQW22}) and the authors introduced quantum twist automorphisms for quantum upper cluster algebras.  
We would like to thank Fan Qin for telling us these after our paper was uploaded to Arxiv. 

The quantum quasi-homomorphism in Definition \ref{def:QQH} is slightly different from the quantum twist automorphism in \cite{KQW22}. The quantum quasi-homomorphism in Definition \ref{def:QQH} is defined for quantum cluster algebras, while quantum twist automorphism in \cite{KQW22} is defined for quantum upper cluster algebras. Definition \ref{def:QQH} is more similar to the non-quantum version of quasi-homomorphisms introduced by Fraser in \cite{Fra16}. 
\end{remark}

Let
\begin{align*}
L=(l_1,\ldots,l_{m-n}), \quad H=(h_1,\ldots,h_{n}), \quad R= \left(\begin{array}{cc} I_{n\times n}&0\\H&L\end{array}\right)\in \Mat_{{m'}\times m}(\ZZ).
\end{align*} 
Denote by $r_i$ the $i$-th column of $R$. Then for any $1\leq i \leq m$,
\begin{equation}\label{eq:map}
f(X^{e_i})=X'^{r_i}.
\end{equation} 

\begin{lemma}\label{Matrixmul}
For a quantum quasi-homomorphism $f: \A_q \to \A'_q$, we have that
$$R^t\Lambda'R=\Lambda.$$
\end{lemma}

\begin{proof}

For any $i,j\in [m]$, we have $X^{e_i}X^{e_j}=q^{\Lambda(e_i,e_j)}X^{e_j}X^{e_i}.$ As $f$ is a $\mathbb C_q$-algebra homomorphism, we obtain
$$f(X^{e_i})f(X^{e_j})=q^{\Lambda(e_i,e_j)}f(X^{e_j})f(X^{e_i}).$$
Then by equalities \eqref{eq:qtorus2} and \eqref{eq:map}, we have $\frac{1}{2}\Lambda'(r_i,r_j)=\Lambda(e_i,e_j)+\frac{1}{2} \Lambda'(r_j,r_i)$. Consequently, $r_i^t\Lambda' r_j=\Lambda'(r_i,r_j)=\Lambda(e_i,e_j)=\lambda_{ij}$ for $i,j\in [1,m]$.  
\end{proof}

For any $i\in\{1,\ldots, n\}$, denote by $R'\in {\rm Mat}_{m'\times m}(\mathbb Z)$ the matrix with the $k$-th column $r'_k$, $k=1,\ldots, m$,  where
\begin{equation*}
r'_{k} = \left\{\begin{array}{ll}
					     e'_i+ R[b_i]_+-[b'_i]_+	& \textrm{ if } k=i~\vspace{1mm}; \\
					     r_k	& \textrm{ if } k \neq i.
					\end{array}\right.
\end{equation*}

Note that if we set $q$ to be one, then the definition degenerates back to the non-quantum case introduced by  \cite{Fra16}.
In particular, the following lemma follows from the commutative case.

\begin{lemma}\label{lem-mm}
For a quantum quasi-homomorphism $f: \A_q \to \A'_q$, we have that
\begin{enumerate}[(1)]
\item $R\t B=\t {B'}$,

\item $R'\mu_i\t{B}=\mu_i\t{B'}.$
\end{enumerate}
\end{lemma}

For any quasi-homomorphism $f:\mathcal A_q(\Sigma)\to \mathcal A_q(\Sigma')$, it is clear that $f$ induces a ring homomorphism $f: \mathbb T(\Sigma)\to \mathbb T(\Sigma')$.

\begin{lemma}\label{Lem-bar}
Let $f:\mathcal A_q\to \mathcal A'_q$ be a quasi-homomorphism. Then for any $a=(a_1,\ldots, a_m)^t\in \mathbb Z^m$, we have that
\begin{align*}
f(X^a)=X'^{Ra}.
\end{align*}
In particular, $f(X^a)$ is bar-invariant.     
\end{lemma}

\begin{proof}
By direct calculation, we have

\begin{equation*}
\begin{array}{rcl}
f(X^a)
&=&f(q^{-\frac{1}{2} \sum_{i<j} a_ia_j\lambda_{ij}}x_1^{a_1}\cdots x_m^{a_m})=q^{-\frac{1}{2} \sum_{i<j} a_ia_j\lambda_{ij}}X'^{a_1r_1}\cdots X'^{a_mr_m}\vspace{2.5pt}\\
&=& q^{-\frac{1}{2} \sum_{i<j} a_ia_j\lambda_{ij}}q^{\frac{1}{2} \sum_{i<j} a_ia_j\Lambda'(r_i,r_j)}X'^{a_1r_1+\cdots +a_mr_m}=X'^{a_1r_1+\cdots +a_mr_m}
\vspace{2.5pt}\\
&=& X'^{Ra},
\end{array}
\end{equation*}
where the last but one equality follows by Lemma \ref{Matrixmul}.
\end{proof}

\begin{proposition}\label{prop-bar}
 Let $f:\mathcal A_q\to \mathcal A'_q$ be a quasi-homomorphism. Then $f$ maps a bar-invariant element in $\mathcal A_q$ to a bar-invariant element in $\mathcal A'_q$, that is, $\overline{f(X)}=f(X)$ if $\overline{X}=X\in \mathcal A_q$. 
\end{proposition}

\begin{proof}
It follows from the second condition in Definition \ref{def:QQH} that $f(x_i), i\in \{1, \ldots, m\}$ are bar-invariant. For any $X\in \mathcal A_q$ such that $\overline{X}=X$, by the quantum Laurent phenomenon, we have $X=\sum_{a} \lambda(a)X^{a}$ for some $\lambda(a)\in \mathbb \CC_q$.
Then we have $f(X)=\sum_{a}\lambda(a) f(X^a)$ and $\overline{f(X)}=\sum_{a}\overline{\lambda(a)} f(X^a)$. As $\overline X=X$, we have $\overline {\lambda(a)}=\lambda(a)$ for each $a$. Then by Lemma \ref{Lem-bar}, we see that $\overline{f(X)}=f(X)$, that is, $f(X)$ is bar-invariant. 
\end{proof}

The following proposition shows that the definition of quasi-homomorphism of quantum cluster algebras is compatible with the mutations.

\begin{proposition}\label{prop:quasi-equ}
Let $f: \mathcal A_q\to \mathcal A'_q$ be a quantum quasi-homomorphism with respect to seeds $\Sigma$ and $\Sigma'$. Then for any sequence $(i_1,i_2,\ldots, i_s)$, $i_j \in \{1,\ldots,n\}$, we have that $f: \mathcal A_q\to \mathcal A'_q$ is a quantum quasi-homomorphism with respect to the seeds $\mu_{i_s}\cdots \mu_{i_1}\Sigma$ and $\mu_{i_s}\cdots \mu_{i_1}\Sigma'$.
\end{proposition}

\begin{proof}
By induction, it suffices to prove the case that $s=1$. Denote $i=i_1$. Then by Lemma \ref{Lem-bar} we have

\begin{equation*}
\begin{array}{rcl}
f(\mu_ix_i)
&=&f(X^{-e_i+[b_i]_+}+X^{-e_i+[-b_i]_+})=X'^{R(-e_i+[b_i]_+)}+X'^{R(-e_i+[-b_i]_+)}\vspace{2.5pt}\\
&=& 
X'^{-e'_i+R[b_i]_+}+X'^{-e'_i+R[-b_i]_+}.
\end{array}
\end{equation*}

By Lemma \ref{lem-mm} (1), we have $[b_i']_+-[-b'_i]_+=b'_i=Rb_i=R[b_i]_+-R[-b_i]_+$. As $\mu_i(x_i')=X'^{-e_i'+[b'_i]_+}+X'^{-e_i'+[-b'_i]_+}$, it follows that
$$f(\mu_ix_i)=(\mu_i(X'))^{e'_i+R[b_i]_+-[b'_i]_+}.$$ 

As $R= \left(\begin{array}{cc} I_{n\times n}&0\\H&L\end{array}\right)$ and $B=B'$, we see that the first $n$ coordinates of $R[b_i]_+-[b'_i]_+$ are $0$. 

For any $i\neq j\in \{1,\ldots, n\}$, as the first $n$ coordinates of $\widetilde h_j$ are $0$,
we have 
$$f(x_j)=X'^{e'_j+\widetilde h_j}=(\mu_i(X'))^{e'_j+\widetilde h_j}.$$

For any $j\in \{n+1,\ldots,m\}$, as the first $n$ coordinates of $\widetilde l_j$ are $0$,
we have 
$$f(x_j)=X'^{\widetilde l_j}=(\mu_i(X'))^{\widetilde l_j}.$$

Denote by $d_j$ and $d'_j$ the $j$-th column of the matrix $\mu_i\widetilde B$ and $\mu_i\widetilde B'$, respectively. Denote $\mu_i(\hat y_j)=(\mu_i(X))^{d_j}$ and $\mu_i(\hat {y'}_j)=(\mu_i(X'))^{{d'}_j}$ for any $j\in \{1, \ldots, n\}$.

Next, we show that $f(\mu_i(\hat y_j))=\mu_i(\hat y'_j)$. If $j=i$, we have $\mu_i(\hat y_i)=\hat y_i^{-1}$ and $\mu_i(\hat {y'}_i)=\hat {y'}_i^{-1}$, and thus 
$$f(\mu_i(\hat y_i))=f(\hat y_i^{-1})=\hat {y'}_i^{-1}=\mu_i(\hat y'_i).$$

If $j\neq i$, we have $\mu_i(\hat y_j)=(\mu_i(X))^{\sum_k d_{kj}}$. Then by Proposition \ref{prop-bar} and Lemma \ref{lem-mm} (2), we have
\begin{align*}
f(\mu_i(\hat y_j)) & =(\mu_i(X'))^{d_
{ij}(e'_i+R[b_i]_+-[b'_i]_+)+\sum_{1\leq k(\neq i)\leq n} d_{kj}(e'_k+\widetilde h_k)+\sum_{n+1\leq k\leq m}d_{kj}\widetilde l_j} \\
& =(\mu_i(X'))^{d'_j}=\mu_i(\hat y'_j).
\end{align*} 

The proof is complete.
\end{proof} 

So roughly speaking, a quasi-homomorphism of quantum cluster algebras is an algebra homomorphism between quantum cluster algebras of the same type which sends a cluster variable to a cluster variable possibly multiplying a Laurent monomial in frozen variables, sends a frozen variable to a Laurent monomial in frozen variables, in a way that keeps the hatted $y$-variables and commutes with the mutations.
On the other hand, from a view point of linear algebra, the following proposition says that a quasi-homomorphism between quantum cluster algebras is equivalent to two matrix equations.

\begin{proposition}\label{prop:equ-quasi-homo}
Let $\x$ and $\x'$ respectively be clusters of $\A_q$ and $\A'_q$ with the same exchange matrix.
A map $f: \x \mapsto \x'$ induces a quasi-homomorphism from $\A_q$ to $\A'_q$ if and only if the following conditions are satisfied,
\begin{enumerate}
\item $f(x_i)=X'^{e'_i+\widetilde h_i}, 1\leq j \leq n,$ where $\t{h}_i=\left(\begin{smallmatrix}0\\ h_i\end{smallmatrix}\right)$ with $h_i\in \ZZ^{{m'}-n}$;

\item $f(x_{n+j})=X'^{\t{l}_j}, 1\leq j \leq m-n,$ where $\t{l}_j=\left(\begin{smallmatrix}0\\ l_j\end{smallmatrix}\right)$ with $l_j\in \ZZ^{{m'}-n}$;

\item $R\t{B}=\t{B'}$, where $L=(l_1,\ldots,l_{m-n})$, $H=(h_1,\ldots,h_{n})$ and $R= \left(\begin{array}{cc} I_{n\times n}&0\\H&L\end{array}\right)$;

\item $R^t\Lambda'R=\Lambda$.
\end{enumerate}
\end{proposition}

\begin{proof}
``Only If Part'': (1) and (2) follow by condition (2) in Definition \ref{def:QQH}. 
(3) follows by Lemma \ref{lem-mm} (1), and (4) follows by Proposition \ref{Matrixmul}.

``If Part": from the proof of Proposition \ref{Matrixmul}, we know that (4) is equivalent to $f(x_i)f(x_j)=q^{\Lambda(e_i,e_j)}f(x_j)f(x_i)$ for any $i,j\in [1,m]$. Therefore, $f$ can be extended to a well-defined algebra homomorphism from $\mathbb{T}$ to $\mathbb{T}'$.
\begin{equation*}
\begin{array}{rcl}
f(\hat y_k)
&=&f(q^{-\frac{1}{2} \sum_{i<j} b_{ik}b_{jk}\lambda_{ij}}x_1^{b_{1k}}\cdots x_m^{b_{mk}})=q^{-\frac{1}{2}\sum_{i<j} b_{ik}b_{jk}\lambda_{ij}}X'^{b_{1k}r_1}\cdots X'^{b_{mk}r_m}\vspace{2.5pt}\\
&=& q^{-\frac{1}{2} \sum_{i<j} b_{ik}b_{jk}\lambda_{ij}}q^{\frac{1}{2} \sum_{i<j} b_{ik}b_{jk}\Lambda'(r_i,r_j)}X'^{b_{1k}r_1+\cdots +b_{mk}r_m}=X'^{b_{1k}r_1+\cdots+ b_{mk}r_m}\vspace{2.5pt}\\
&=&X'^{Rb_k}=X'^{b'_k}=\hat y'_k,
\end{array}
\end{equation*}
where the last but one equality follows from (3).

Furthermore, by a similar argument used in Proposition \ref{prop:quasi-equ}, $f$ maps any cluster variable of $\mathcal A_q$ to a cluster variable possibly multiplying a coefficient monomial of $\mathcal A'_q$. In particular, $f$ induces a quasi-homomorphism from $\mathcal A_q$ to $\mathcal A'_q$.
\end{proof}

To introduce a concept of quasi-isomorphism of quantum cluster algebras, we need the following proportional relation, which is a quantum analogy of the one introduced by Fraser \cite{Fra16}.
For elements $x,y$ in a quantum cluster algebra $\mathcal{A}_q$, we say that $x$ \emph{is proportional to} $y$, writing $x \propto y$, if $x = q^{\frac{\ell}{2}}p y$ for some $\ell\in \mathbb Z$ and some coefficient monomial $p\in \mathbb P$. Likewise, let $\mathcal{A}_q$ and $\mathcal{A}'_q$ be a pair of quantum cluster algebras, and let $f_1$ and $f_2$ be two quasi-homomorphisms from $\mathcal A_q$ to $\mathcal A'_q$, we say that $f_1$ \emph{is proportional to} $f_2$, and write $f_1 \propto f_2$, if $f_1(x) \propto f_2(x)$ for every cluster variable $x \in \mathcal{A}_q$.

A quasi-homomorphism $f \colon \mathcal{A}_q \to \mathcal{A}'_q$ is called a \emph{quasi-isomorphism} if there is a quasi-homomorphism $g \colon \mathcal{A}'_q \to \mathcal{A}_q$ such that the composite $g \circ f $ is proportional to the identity map on $\mathcal{A}_q$ and the composite $f \circ g$ is proportional to the identity map on $\mathcal{A}'_q$. A \emph{quasi-automorphism} is a quasi-isomorphism from a quantum cluster algebra to itself.
Then we are able to consider the \emph{quasi-automorphism group} ${\rm QAut}(\A_q)$ of a quantum cluster algebra $\A_q$, which consists of the equivalence classes of quasi-automorphisms on $\A_q$ up to proportionality.

Note that for a classical (non-quantum) cluster algebra, Assem, Schiffler, and Shramchenko introduced cluster automorphisms in \cite{ASS}. Similarly, one may consider the quantum version, a \emph{quantum cluster automorphism}, that is, a $\CC_q$-algebra automorphism of a quantum cluster algebra which sends a cluster to a cluster, sends a cluster variable to a cluster variable (without multiplying frozen variables), and sends a frozen variable to a frozen variable. Then it is not hard to see that a quantum cluster automorphism is a special quantum quasi-automorphism, and a quantum quasi-homomorphism $f$ is a quantum cluster automorphism if and only if $(\t{B},\Lambda)=(\t{B'},\Lambda')$.

It would be interesting to compute the \emph{automorphism group} ${\rm Aut}(\A_q)$ of a quantum cluster algebra, which consists of all the quantum cluster automorphisms, and compare it with the group ${\rm QAut}(\A_q)$.
This is in some sense trivial for the quantum cluster algebras arising from the surfaces.
More precisely, let $\A$ be the (classical non-quantum) cluster algebra arising from a unpunctured surface $(S,M)$ with coefficients arising from the boundary segments, and let $\A_{\rm triv}$ be the associated cluster algebra with trivial coefficients. Since ${\rm Aut}(\A)\cong {\rm Aut}(\A_{\rm triv})\cong {\rm MCG}(S,M)$ \cite{BS15,CZ16}, where ${\rm MCG}(S,M)$ is the mapping class group of the surface, triangulations $T$ and $T'$ are homotopic if and only if the associated matrices $\t{B}$ and $\t{B'}$ coincide. Moreover, the Lambda matrix of the associated quantum surface cluster algebra $\mathcal A_q$ comes from the triangulation in the sense of Muller \cite{M16}, thus the matrices $\Lambda$ and $\Lambda'$ arising from $T$ and $T'$ coincide: $\Lambda=\Lambda'$. Therefore we have an isomorphism ${\rm Aut}(\A_q)\cong {\rm Aut}(\A)$. 
On the other hand, it is proved in \cite{CS19} that  ${\rm QAut}(\A)\cong {\rm Aut}(\A)$.
By using the fact that the Lambda matrix is arising from the triangulation, the method used in \cite{CS19} can be extended to the quantum case, and we have an isomorphism of groups: ${\rm QAut}(\A_q)\cong {\rm Aut}(\A_q)$. To sum up, for (almost all) the surface cluster algebras, both for the quantum case and the non-quantum case, the groups we mentioned above are all isomorphic.

\section{Braid group action on quantum Grassmannian cluster algebras} \label{sec:braid group action on quantum Grassmannian cluster algebras}

In this section, we introduce a braid group action on the quantum Grassmannian cluster algebra $\CC_q[\Gr(k,n)]$.
We collect some background in the following, \cite{TT91, LZ98, BG02, Sco05, Lau06,Lau10, GL14}. For $\CC_q[\Gr(k,n)]$, we assume that $q$ is not a root of unity.

\subsection{Background on the quantum Grassmannian coordinate ring $\CC_q[\Gr(k,n)]$} \label{subsec:quantum Grassmannian coordinate ring}
The quantum matrix algebra $\CC_q[M(k,n)]$ is a $\CC_q$-algebra generated by $\{x_{ij} \mid 1 \leq i \leq k, 1 \leq j \leq n\}$ subject to the following relations.
\begin{alignat*}{2}
& x_{ij} x_{il}  = q  x_{il} x_{ij},  && j < l, \\
& x_{ij} x_{rj}  = q x_{rj} x_{ij}, && i<r, \\
& x_{ij} x_{rl} = x_{rl} x_{ij}, && i<r, j>l, \\
& x_{ij} x_{rl} = x_{rl} x_{ij} + (q - q^{-1})x_{il}x_{rj}, \quad  && i < r, j < l.
\end{alignat*}

Let $m,n \in \mathbb{Z}_{>0}$. For any $I \subset [1,m]$, $J \subset [1,n]$ with $|I|=|J|=l >0$, a quantum minor $\Delta_{I,J}$ is defined by
\begin{align*}
\Delta_q^{I,J} = \sum_{\sigma \in S_l} (-q)^{\ell(\sigma)} x_{i_1, j_{\sigma(1)}} \cdots x_{i_l, j_{\sigma(l)}},
\end{align*}
where $\{i_1 < \cdots < i_l\} = I$, $\{j_1< \cdots < j_l\}=J$, and $\ell(\sigma)$ is the length of the permutation $\sigma$.

For the Grassmannian $\Gr(k,n)$, its quantized coordinate ring $\CC_q[\Gr(k,n)]$
is the subalgebra of the quantum matrix algebra $\CC_q[M(k,n)]$ generated by the quantum Pl\"{u}cker coordinates
\[\Delta_q^J := \Delta_q^{I,J},\]
where $I= \{1, \ldots, k\}$ and $J \subset [1,n]$ with $|J|=k$, see \cite{TT91, GL14}. We assume that the frozen quantum Pl\"{u}cker coordinates $\Delta_q^J$ are invertible, where $J$ consists of $k$ consecutive integers (up to cyclic shift).

It is shown in \cite{LZ98} that two quantum Pl\"{u}cker coordinates $\Delta_q^I$, $\Delta_q^J$ quasi-commute if and only if $I, J$ are weakly separated. We recall the definition of ``weakly separated'' in the following.

\begin{definition} \label{def:weakly separted k-subsets}
Given two $k$-subsets $I$ and $J$ of $[1,n]$, denote by $\min(J)$ the minimal element in $J$ and by $\max(I)$ the maximal element in $I$, we write $I \prec J$ if $\max(I)<\min(J)$. The sets $I$ and $J$ are called weakly separated if at least one of the
following two conditions holds:
\begin{enumerate}
\item $J-I$ can be partitioned into a disjoint union $J-I = J' \sqcup J''$ so that $J' \prec I-J \prec J''$;

\item $I-J$ can be partitioned into a disjoint union $I-J = I' \sqcup I''$ so that $I' \prec J-I \prec I''$.
\end{enumerate}
\end{definition}

In {\cite[Theorem 2]{Sco05}}, Scott computed quasi-commutation relations on weakly separated subsets.

\begin{theorem} [{ \cite[Theorem 2]{Sco05} }] \label{thm:scott}
Suppose $I$ and $J$ are weakly separated $k$-subsets of $[1,n]$. Then we have $\Delta_q^I \Delta_q^J=q^{\Lambda(\Delta_q^I, \Delta_q^J)} \Delta_q^J \Delta_q^I$, where
\begin{align*}
\Lambda(\Delta_q^I, \Delta_q^J)  = \begin{cases} |J''| - |J'|, & \text{ if (1) in Definition \ref{def:weakly separted k-subsets} is satisfied}, \\
|I'| - |I''|, & \text{ if (2) in Definition \ref{def:weakly separted k-subsets} is satisfied}.
\end{cases}
\end{align*}
\end{theorem}

Every maximal weakly separated collection of $k$-subsets of $[1,n]$ determines a seed in the cluster algebra $\CC[\Gr(k,n)]$, where the cluster in the seed is a set of Pl\"{u}cker coordinates parameterized by the subsets in the collection, see \cite{LZ98, Pos06, Sco06, OPS15}. According to Theorem 6.1 in \cite{BZ05} and Theorem 7.6 in \cite{GL14}, the exchange graphs of $\CC[\Gr(k,n)]$ and $\CC_q[\Gr(k,n)]$ are isomorphic. Together with the results in \cite{LZ98, Pos06, Sco05, Sco06, OPS15}, we have that each maximal weakly separated collection also determines a seed in the quantum cluster algebra $\CC_q[\Gr(k,n)]$, where the Lambda matrix is given by the quasi-commutation relations of the associated Pl\"{u}cker coordinates in Theorem \ref{thm:scott}.

\subsection{Quantum cluster Laurent monomials and quantum cluster mutations on $\CC_q[\Gr(k,n)]$} \label{subsec:commutative quantum cluster monomials}

Let $\{x_1,\cdots, x_m\}$ be a quantum cluster. For any $k_1,\cdots, k_m\in \mathbb Z$, we use the notation $[x_1^{k_1}\cdots x_m^{k_m}]$ to denote the quantum cluster Laurent monomial $X^{k_1e_1+\cdots +k_me_m}$. Note that changing the order of $x_1, \ldots, x_m$ in $[x_1^{k_1}\cdots x_m^{k_m}]$ does not affect the result. 

Quantum Grassmannian cluster algebras was studied in \cite{GL11, GL14}. We recall their results in the case of $\CC_q[\Gr(3,6)]$. We compute quantum cluster variables of $\CC_q[\Gr(3,6)]$, starting from the cluster in Figure \ref{fig:initial cluster Gr36}.

\begin{figure} 
\begin{tikzpicture}[scale=0.39]
     \node at (0,0) (v00) {\fbox{$\Delta_q^{123}$}};
     \node at (0,-4) (v10) {$\Delta_q^{124}$};
     \node at (0,-8) (v20) {$\Delta_q^{125}$};
     \node at (0,-12) (v30) {\fbox{$\Delta_q^{126}$}};

     \node at (4,-4) (v11) {$\Delta_q^{134}$};
     \node at (4,-8) (v21) {$\Delta_q^{145}$};
     \node at (4,-12) (v31) {\fbox{$\Delta_q^{156}$}};

     \node at (8,-4) (v12) {\fbox{$\Delta_q^{234}$}};
     \node at (8,-8) (v22) {\fbox{$\Delta_q^{345}$}};
     \node at (8,-12) (v32) {\fbox{$\Delta_q^{456}$}};

     \draw[->] (v10)--(v00);
     \draw[->] (v20)--(v10);
     \draw[->] (v30)--(v20);

     \draw[->] (v21)--(v11);
     \draw[->] (v31)--(v21);

     \draw[->] (v22)--(v12);
     \draw[->] (v32)--(v22);

     \draw[->] (v11)--(v10);
     \draw[->] (v12)--(v11);

     \draw[->] (v21)--(v20);
     \draw[->] (v22)--(v21);

     \draw[->] (v31)--(v30);
     \draw[->] (v32)--(v31);

     \draw[->] (v10)--(v21);
     \draw[->] (v21)--(v32);

     \draw[->] (v20)--(v31);

     \draw[->] (v11)--(v22);

\end{tikzpicture} 
\caption{An initial cluster for $\CC_q[\Gr(3,6)]$.}
\label{fig:initial cluster Gr36}
\end{figure}

We have relations
\begin{align*}
& \Delta_q^{124}   \Delta_q^{123} = \frac{1}{q} \Delta_q^{123}   \Delta_q^{124}, \quad \Delta_q^{124}   \Delta_q^{145} = q \Delta_q^{145}   \Delta_q^{124}, \quad \Delta_q^{124}   \Delta_q^{125} = q \Delta_q^{125}   \Delta_q^{124}, \\
& \Delta_q^{124}   \Delta_q^{134} = q \Delta_q^{134}   \Delta_q^{124}, \quad \Delta_q^{125}   \Delta_q^{134} = \Delta_q^{134}   \Delta_q^{125}, \quad \Delta_q^{123}   \Delta_q^{145} = q^2 \Delta_q^{145}   \Delta_q^{123}.
\end{align*}

Mutating at $\Delta_q^{124}$, we have that
\begin{align*}
(\Delta_q^{124})' = [(\Delta_q^{124})^{-1}  \Delta_{q}^{125} \Delta_{q}^{134}] + [\Delta_{q}^{123} (\Delta_q^{124})^{-1}  \Delta_{q}^{145}].
\end{align*}
On the other hand, we have the relation
\begin{align*}
\Delta_q^{124} \Delta_{q}^{135} = q \Delta_{q}^{125}  \Delta_{q}^{134} + \frac{1}{q}  \Delta_{q}^{123} \Delta_{q}^{145}  = q [\Delta_{q}^{125} \Delta_{q}^{134}] + [\Delta_{q}^{123} \Delta_{q}^{145}].
\end{align*}
Therefore
\begin{align*}
\Delta_q^{135} = (\Delta_q^{124})^{-1}   (q [\Delta_{q}^{125} \Delta_{q}^{134}] + [\Delta_{q}^{123} \Delta_{q}^{145})] =  [(\Delta_q^{124})^{-1}  \Delta_{q}^{125} \Delta_{q}^{134}] +  [\Delta_{q}^{123} (\Delta_q^{124})^{-1} \Delta_{q}^{145}].
\end{align*}
It follows that $(\Delta_q^{124})' = \Delta_q^{135}$. Using quantum mutations, we see that the quantum cluster variables are:
\begin{align*}
& \Delta_q^{124}, \ \Delta_q^{125}, \ \Delta_q^{134}, \ \Delta_q^{135}, \ \Delta_q^{136}, \ \Delta_q^{145}, \ \Delta_q^{146}, \ \Delta_q^{235}, \ \Delta_q^{236}, \ \Delta_q^{245}, \ \Delta_q^{246}, \ \Delta_q^{256}, \ \Delta_q^{346}, \ \Delta_q^{356},  \\
& q^{-\frac{3}{2}}( \Delta_q^{124} \Delta_q^{356} - \frac{1}{q} \Delta_q^{123} \Delta_q^{456} ), \ q^{-\frac{1}{2}}( \Delta_q^{145} \Delta_q^{236} - \frac{1}{q^2} \Delta_q^{123} \Delta_q^{456} ).
\end{align*}

\subsection{A bijection between cluster Laurent monomials in Grassmannian cluster algebra and the corresponding quantum Grassmannian cluster algebra} \label{subsec:the bijection varphi between cluster monomials and quantum cluster monomials}

We define a map $\varphi$ from the set of cluster variables of $\CC[\Gr(k,n)]$ to the set of cluster variables in $\CC_q[\Gr(k,n)]$ as follows. First for $\Delta^J$ ($\Delta^J$ can be a frozen variable) in the initial cluster, we define $\varphi(\Delta^J) = \Delta_q^J$. We use induction to define $\varphi(x)$ for other cluster variables in $\CC[\Gr(k,n)]$. We apply the same mutation sequence to $\CC[\Gr(k,n)]$ and $\CC_q[\Gr(k,n)]$. At each step of mutation at some vertex with cluster variable $x$, we define $\varphi(x') = (\varphi(x))'$, where  $x'$ is the cluster variable of $\CC[\Gr(k,n)]$ obtained by mutation at $x$ and $(\varphi(x))'$ is the cluster variable obtained by mutation at $\varphi(x)$ in $\CC_q[\Gr(k,n)]$. For a cluster variable or frozen variable $x$, we define $\varphi(x^{-1})=(\varphi(x))^{-1}$. According to Theorem 6.1 in \cite{BZ05} and Theorem 7.6 in \cite{GL14}, the exchange graph of the cluster algebra $\CC[\Gr(k,n)]$ and the exchange graph of the quantum cluster algebra $\CC_q[\Gr(k,n)]$ are isomorphic.  
The map $\varphi$ is clearly a bijection between the set of cluster variables in $\CC[\Gr(k,n)]$ and the set of cluster variables in $\CC_q[\Gr(k,n)]$. Therefore every cluster variable or frozen variable in $\CC_q[\Gr(k,n)]$ is of the form $\varphi(x)$ for some cluster variable $x$ in $\CC[\Gr(k,n)]$.

For a cluster monomial $x_1\cdots x_r$ in $\CC[\Gr(k,n)]$ and for $a_1, \ldots, a_r \in \ZZ$, we define 
\begin{align}\label{equ:extend}
\varphi(x_1^{a_1} \cdots x_r^{a_r}) = [\varphi(x_1)^{a_1} \cdots \varphi(x_r)^{a_r}].
\end{align} 

In \cite{GL11, GL14}, it is shown that $\CC_q[\Gr(k,n)]$ has a quantum cluster algebra structure. The exchange relations in $\CC_q[\Gr(k,n)]$ are of the form $x'_i = [x^{-1}_i \prod_{j \to i} x_j] + [x^{-1}_i \prod_{i \to j} x_j]$. 
Moreover, the extended exchange matrices of $\CC_q[\Gr(k,n)]$ and $\CC[\Gr(k,n)]$ are the same. 
This implies that for any exchange relation $\tilde{x}'_i = \tilde{x}^{-1}_i \prod_{j \to i} \tilde{x}_j + \tilde{x}^{-1}_i \prod_{i \to j} \tilde{x}_j$ in $\CC[\Gr(k,n)]$, if we replace $\tilde{x}_j$'s by $\varphi(\tilde{x}_j)$, we obtain an exchange relation $\varphi(\tilde{x}_i)' = [\varphi(\tilde{x}_i)^{-1} \prod_{j \to i} \varphi(\tilde{x}_j)] + [\varphi(\tilde{x}_i)^{-1} \prod_{i \to j} \varphi(\tilde{x}_j)]$ in $\CC_q[\Gr(k,n)]$.  

\begin{remark}
The algebras $\CC[\Gr(k,n)]$, $\CC_q[\Gr(k,n)]$ are not isomorphic. Therefore $\varphi$ cannot be extended to an algebra isomorphism between $\CC[\Gr(k,n)]$ and $\CC_q[\Gr(k,n)]$.
\end{remark} 

\subsection{Fraser's braid group action on Grassmannian cluster algebras} \label{subsec:Fraser braid group action}
We now recall Fraser's braid group action on $\CC[\Gr(k,n)]$. Let $V$ be a $k$-dimensional complex vector space. An $n$-tuple of vectors $(v_1, \ldots, v_n)$ is called consecutively generic if every cyclically consecutive $k$-tuple of vectors is linearly independent, i.e., $\det( v_{i+1}, \ldots, v_{i+k} ) \ne 0$ for $i=1,\ldots,n$ where the indices are treated modulo $n$. Denote by $(V^n)^{\circ} \subset V^n$ the quasi-affine variety consisting of consecutively generic $n$-tuples. 

Let $d={\rm gcd}(k,n)$. For $i \in [1,d-1]$, the map $\sigma_i: (V^n)^{\circ} \to (V^n)^{\circ}$ is defined as follows, see Definition 5.2 and Equation (18) in Remark 5.6 of \cite{Fra17}. Divide an element $(v_1, \ldots, v_n)$ in $(V^n)^{\circ}$ into $\frac{n}{d}$ windows: $[v_{1+jd}, \ldots, v_{(j+1)d}]$, $j \in [0, \frac{n}{d}-1]$. The map $\sigma_i$ sends the first window to $[v_1, \ldots, v_{i-1}, v_{i+1}, w_1, v_{i+2}, \ldots, v_d]$, where 
\begin{align*}
    w_1 = \frac{\det(v_i, v_{i+2}, \ldots, v_{i+k})}{\det( v_{i+1}, v_{i+2}, \ldots, v_{i+k} )} v_{i+1} - v_{i}.
\end{align*}
The $\ell$th window is defined by the same recipe by $d$-periodically augmenting indices. The map $\sigma_i^{-1}$ us defined similarly. In Theorem 5.3 of \cite{Fra17}, Fraser proved that the maps $\sigma_i$, $\sigma_i^{-1}$ are inverse regular automorphisms on $(V^n)^{\circ}$. The pullbacks $\sigma_i^*$, $(\sigma_i^{-1})^*$ are inverse quasi-automorphisms on $\CC[\Gr(k,n)]$ (our notation $\Gr(k,n)$ in this paper has the same meaning as $\widetilde{\Gr}(k,n)$ in \cite{Fra17}. It is the affine cone over the Grassmannian of $k$-dimensional subspaces in $\CC^n$. With slight abuse of notation, we write the Zariski-open subset $\widetilde{\Gr}^{\circ}(k,n)$ of $\widetilde{\Gr}(k,n)$ cut out by the non-vanishing of the frozen variables as $\Gr(k,n)$).

With a slight abuse of notation, we also write $\sigma_i^*$, $(\sigma_i^{-1})^*$ as $\sigma_i$, $\sigma_i^{-1}$ respectively. 

\subsection{The maps $\sigma_i$ and the main results} \label{subsec:the maps sigmai and main results}

We now define a map $\sigma_i: \CC_q[\Gr(k,n)] \to \CC_q[\Gr(k,n)]$ for any $i \in [1, d-1]$, $d = {\rm gcd}(k,n)$. 

It is shown in Theorem 5.3 of \cite{Fra17} that in $\CC[\Gr(k,n)]$, every $\sigma_i$ ($i \in [1, d-1]$, $d={\rm gcd}(k,n)$) is a quasi-automorphism. Therefore in $\CC[\Gr(k,n)]$, for every cluster variable or a frozen variable $x$, $\sigma_i(x) = y_1 y_2 \cdots y_r \tilde{x}$, where $y_i$ is a frozen variable or the inverse of a frozen variable, and $\tilde{x}$ is a cluster variable. In $\CC_q[\Gr(k,n)]$, we define $\sigma_i(\varphi(x))$ to be $[\varphi(y_1)\varphi(y_2)\cdots \varphi(y_r) \varphi(\tilde{x})]$. As explained in Section \ref{subsec:the bijection varphi between cluster monomials and quantum cluster monomials}, by Theorem 6.1 in \cite{BZ05} and Theorem 7.6 in \cite{GL14}, the exchange graph of the quantum Grassmannian cluster algebra $\CC_q[\Gr(k,n)]$ is isomorphic to the exchange graph of the corresponding Grassmannian cluster algebra $\CC[\Gr(k,n)]$. Therefore there is one to one correspondence between cluster monomials in $\CC_q[\Gr(k,n)]$ and $\CC[\Gr(k,n)]$. Thus every cluster monomial $x$ in $\CC_q[\Gr(k,n)]$ has a unique image $\sigma_i(x)$. 

We will prove in Lemma \ref{lem:sigma_i preserves quasicommutative relations} that the map $\sigma_i$ preserve the quasi-commutative relations between quantum Pl\"{u}cker coordinates which are in the same cluster, and prove in Lemma \ref{lem:sigma_i preserves exchange relations} that $\sigma_i$ preserves exchange relations. Up to Conjecture \ref{conj:sigmai preserves r term plucker relations r is 4 or more}, the map $\sigma_i$ is well-defined. 
  
We define $\sigma_i(\varphi(x^{-1})) = (\sigma_i(\varphi(x)))^{-1}$. We extend $\sigma_i$ to $\CC_q[\Gr(k,n)]$ by defining $\sigma_i(x y) = \sigma_i(x) \sigma_i(y)$ and $\sigma_i(ax+by)=a\sigma_i(x)+b\sigma_i(y)$ for any cluster monomials $x, y$ and $a, b \in \mathbb{C}_q$. 

For a cluster monomial $x$ in $\CC_q[\Gr(k,n)]$, we define $\sigma_i^{-1}(x)$ in the same way as above by replacing $\sigma_i$ by $\sigma_i^{-1}$. 

In the following, we describe the map $\sigma_i$ explicitly on some particular initial cluster. 
We recall a nice initial cluster $\x$ in $\CC_q[\Gr(k,n)]$, which is used frequently when consider Grassmannian cluster algebras, see for example \cite{Fra17} and the quantum version in \cite{GL14}. The cluster is defined as
\begin{equation}\label{eq:cluster}
\x=\{\Delta_q^{\widehat{I}} \colon \ \widehat{I}=[1,a] \cup [a+b+1,b+k], \ 0 \leq a \leq k , \ 0 \leq b \leq n-k \}.
\end{equation}
The quiver of the associated matrix of $\x$
is given in Figure \ref{fig:quiver-cluster}, where the coordinate at position $(a,b)$ is
\begin{equation}\label{eq:coordinate}
\Delta_q^{\widehat{I}_{(a,b)}},~~~ \widehat{I}_{(a,b)}=[1,a] \cup [a+b+1,b+k].
\end{equation}

\begin{figure}[H]
\begin{align*}
\scalebox{0.56}{
\xymatrix{
(0,n-k)\ar[dr] & (1,n-k)\ar[dr] & (2,n-k)\ar[dr] & \cdots\ar[dr] & (k-1,n-k)             \\
(0,n-k-1)\ar[dr] & (1,n-k-1)\ar[dr]\ar[u]\ar[l] & (2,n-k-1)\ar[dr]\ar[u]\ar[l] & \cdots\ar[l]\ar[dr] & (k-1,n-k-1)\ar[u]\ar[l]               \\
(0,n-k-2)\ar[dr] & (1,n-k-2)\ar[dr]\ar[u]\ar[l] & (2,n-k-2)\ar[dr]\ar[u]\ar[l] & \cdots\ar[l]\ar[dr]  & (k-1,n-k-2)\ar[u]\ar[l]               \\
{\begin{array}{c}  \cdot \\  \cdot \\  \cdot \end{array}}\ar[dr] &
{\begin{array}{c}  \cdot \\  \cdot \\  \cdot \end{array}}\ar[u]\ar[dr] &
{\begin{array}{c}  \cdot \\  \cdot \\  \cdot \end{array}}\ar[u]\ar[dr] &
{\begin{array}{c}  \cdot \\  \cdot \\  \cdot \end{array}}\ar[dr] &
{\begin{array}{c}  \cdot \\  \cdot \\  \cdot \end{array}}\ar[u] \\
(0,3)\ar[dr] & (1,3)\ar[dr]\ar[u]\ar[l] & (2,3)\ar[dr]\ar[u]\ar[l] & \cdots\ar[l]\ar[dr]  & (k-1,3)\ar[u]\ar[l]               \\
(0,2)\ar[dr] & (1,2)\ar[dr]\ar[u]\ar[l] & (2,2)\ar[dr]\ar[u]\ar[l] & \cdots\ar[l]\ar[dr]  & (k-1,2)\ar[u]\ar[l]               \\
(0,1) & (1,1)\ar[u]\ar[l] & (2,1)\ar[u]\ar[l] & \cdots\ar[l]  & (k-1,1)\ar[u]\ar[l]               \\
(0,0)\ar[urrrr]&&&&\\
} }
\end{align*}
\caption{The quiver of the cluster $\xx$ in \eqref{eq:cluster}.}
\label{fig:quiver-cluster}
\end{figure}

The frozen variables of $\CC_q[\Gr(k,n)]$ correspond to $k$-subsets of the form $[j, j+k-1]$ for some $j \in [1, n-k+1]$ or $[1, j] \cup [n-k+j+1, n]$ for some $j \in [1, k-1]$.

Let $d={\rm gcd}(k,n)$. Note that $d=1$ is trivial for our purpose, so we always assume that $d\geq 2$. For $i \in [1,d-1]$, we now compute the image of a particular cluster under the map $\sigma_i$. For $\widehat{I}$ in $\Delta_q^{\widehat{I}} \in {\bf x}$, we set
\begin{equation}
\label{eq:Imod}
I =
\begin{cases}
{\widehat{I}}, & \text{ if ${\widehat{I}}$ is frozen}; \\
{\widehat{I}}, & \text{ if $a+b \not \equiv i \mod d$}; \\
[1,a-1] \cup \{a+b\} \cup [a+b+2,b+k+1], & \text{ if $a+b \equiv i \mod d$.}
\end{cases}
\end{equation}
Then these modified subset is still a weakly separated collection and the associated set $\x(i)=\{\Delta_q^{I}\}$ is a cluster, see Section 5.1 in \cite{Fra17}.

Let $\mathfrak{S}_n$ denote the symmetric group on $n$ symbols. Define $\overline{\sigma_i} \in \mathfrak{S}_n$ as the product of commuting transpositions
$\overline{\sigma}_i = \prod_{j=0}^{\frac{n}{d}-1}(jd+i,jd+i+1).$ This permutation acts by switching adjacent numbers that are equivalent to $i$ and $i+1 \mod d$.  

For $\Gr(k,n)$, in the case that $n<a+k-1 \le n+k-1$, the interval $[a,a+k-1]$ means $[1,2,\ldots, a+k-n-1, a,a+1, \ldots, n]$. For example, in the case of $\Gr(3,6)$, $[5,6,7]$ means $[1,5,6]$.  The following lemma follows from \cite[Section 5]{Fra17} and the definition of $\sigma_i$ on $\CC_q[\Gr(k,n)]$ in the beginning of this subsection.
\begin{lemma}\label{lem:sigma on an initial cluster}
The map $\sigma_i$ on $\x(i)$ is given by:
\begin{equation}
\label{eq:sigma-map}
\scalemath{0.9}{
\sigma_i(\Delta_q^{I}) =
\begin{cases}
[\Delta_q^{[j-1,j+k-2]}(\Delta_q^{[j,j+k-1]})^{-1}\Delta_q^{[j+1,j+k]}], &
\text{ if $I=[j,j+k-1]$ with $j \equiv i+1\mod d$}; \\
\Delta_q^{I},& \text{ if $I=[j,j+k-1]$ with $j \not \equiv i+1\mod d$}; \\
\Delta_q^{\overline{\sigma}_i(I)},& \text{otherwise}.
\end{cases} }
\end{equation} 
\end{lemma}

\begin{example}
In $\CC_q[\Gr(3,6)]$, $\sigma_1(\Delta_q^{124}) = \Delta_q^{125}$, and 
\begin{align*}
\sigma_1(\Delta_q^{234}) & =[\Delta_q^{123}(\Delta_q^{234})^{-1}\Delta_q^{345}] = \Delta_q^{123} (\Delta_q^{234})^{-1} \Delta_q^{345}  = \Delta_q^{345} (\Delta_q^{234})^{-1} \Delta_q^{123} \\
& = \frac{1}{q} (\Delta_q^{234})^{-1} \Delta_q^{123} \Delta_q^{345} = \frac{1}{q}\Delta_q^{123} \Delta_q^{345} (\Delta_q^{234})^{-1} \\
& = q (\Delta_q^{234})^{-1} \Delta_q^{345} \Delta_q^{123} = q \Delta_q^{345} \Delta_q^{123} (\Delta_q^{234})^{-1}.
\end{align*}  
\end{example}

\begin{lemma} \label{lem:sigma_i preserves quasicommutative relations}
Let $i \in [d-1]$ and let $x_1, x_2$ be any pair of cluster variables (including frozen variables) in a cluster. Then for $\widetilde{x}_1 \in \{  x_1, x_1^{-1} \}$, $\widetilde{x}_2 \in \{x_2, x_2^{-1} \}$, we have that $\widetilde{x}_1 \widetilde{x}_2=q^c\widetilde{x}_2 \widetilde{x}_1$ for some $c$ implies that $\sigma_i(\widetilde{x}_1) \sigma_i(\widetilde{x}_2) =q^c \sigma_i(\widetilde{x}_2) \sigma_i(\widetilde{x}_1)$, and $\sigma_i^{-1}(\widetilde{x}_1) \sigma_i^{-1}(\widetilde{x}_2) =q^c \sigma_i^{-1}(\widetilde{x}_2) \sigma_i^{-1}(\widetilde{x}_1)$.
\end{lemma}

\begin{proof}
We will prove the statement for $\sigma_i$. The proof of the result for $\sigma_i^{-1}$ is the similar. 

Fraser \cite{Fra17} has proved that in non-quantum case, $\sigma_i$ sends a cluster to a cluster up to frozen variables. Therefore by the definition of $\sigma_i$ in quantum case in the beginning of this subsection, we have that for two cluster variables $x_1, x_2$ in the same cluster, $\sigma_i(x_1)$, $\sigma_i(x_2)$ are cluster variables (possibly multiply by Laurent monomials of frozen variables). 

It suffices to prove the result for $x_1, x_2$ in one particular cluster. Indeed, suppose that the result is true for the cluster $\x=(x_1, \ldots, x_m)$. Let $x_r'$ be the quantum cluster variable obtained by mutating $x_r$ and let $\x' = (x_1, \ldots, x_r', \ldots, x_m)$. Then $x_r' = [x_r^{-1} x_{j_1} \cdots x_{j_t}]+ [x_r^{-1} x_{e_1} \cdots x_{e_s}]$, where   
$s, t, j_l, e_l$ are some integers. Therefore $x_r' = q^a x_r^{-1}   x_{j_1}  \cdots  x_{j_t} + q^b x_r^{-1}  x_{e_1}  \cdots  x_{e_s}$ for some $a, b$. Suppose that $x_f   x_r' = q^c x_r'   x_f$ for some $c$ and some $x_f\in \x$. We want to show that $\sigma_i(x_f) \sigma_i(x_r') = q^c \sigma_i(x_r') \sigma_i(x_f)$. We have 
\begin{align*}
x_f x_r' = x_f ( q^a x_r^{-1}   x_{j_1}  \cdots  x_{j_t} + q^b x_r^{-1}  x_{e_1}  \cdots  x_{e_s} ).
\end{align*}
Since $x_f, x_r, x_{j_l}, x_{e_l}$ are in the same cluster, we have 
\begin{align*}
x_f x_r' = (q^{a+d} x_r^{-1}   x_{j_1}  \cdots  x_{j_t} + q^{b+d'} x_r^{-1}  x_{e_1}  \cdots  x_{e_s} )   x_f,
\end{align*}
for some $d, d'$. Since we have assumed that $x_f   x_r' = q^c x_r'   x_f$, we have $d=d'=c$. Since we have assumed that the result is true for the cluster $\x=(x_1, \ldots, x_m)$, we have that
\begin{align*}
& \sigma_i(x_f) \sigma_i(x_r') = \sigma_i(x_f) \sigma_i( q^a x_r^{-1}   x_{j_1}  \cdots  x_{j_t} + q^b x_r^{-1}  x_{e_1}  \cdots  x_{e_s} ) \\
& = \sigma_i(x_f) ( q^a \sigma_i(x_r^{-1})   \sigma_i(x_{j_1})  \cdots  \sigma_i(x_{j_t}) + q^b \sigma_i(x_r^{-1})  \sigma_i(x_{e_1})  \cdots  \sigma_i(x_{e_s}) ) \\
& = ( q^{a+c} \sigma_i(x_r^{-1})   \sigma_i(x_{j_1})  \cdots  \sigma_i(x_{j_t})  + q^{b+c} \sigma_i(x_r^{-1})  \sigma_i(x_{e_1})  \cdots  \sigma_i(x_{e_s}) ) \sigma_i(x_f).
\end{align*}
Therefore $\sigma_i(x_f) \sigma_i(x_r') = q^c \sigma_i(x_r') \sigma_i(x_f)$. It follows that $x_f^{\pm 1}   {x_r'}^{\pm 1} = q^c {x_r'}^{\pm 1}   x_f^{\pm 1}$ for some $c$ implies that $\sigma_i(x_f^{\pm 1}) \sigma_i({x_r'}^{\pm 1}) = q^c \sigma_i({x_r'}^{\pm 1}) \sigma_i(x_f^{\pm 1})$.

We now prove the following: for cluster variables (including frozen variables) $\Delta^{J}_q$, $\Delta^{K}_q$ in $\x(i)$ (see Equation (\ref{eq:Imod})), we have that $\Delta^{J}_q \Delta^{K}_q = q^{\lambda_{J,K}} \Delta^{K}_q \Delta^{J}_q$ implies that $\sigma_i(\Delta^{J}_q) \sigma_i(\Delta^{K}_q) = q^{\lambda_{J,K}} \sigma_i(\Delta^{K}_q) \sigma_i(\Delta^{J}_q)$. From this, it follows that $(\Delta^{J}_q)^{\pm 1} (\Delta^{K}_q)^{\pm 1} = q^{\lambda_{J,K}} (\Delta^{K}_q)^{\pm 1} (\Delta^{J}_q)^{\pm 1}$ implies that $\sigma_i((\Delta^{J}_q)^{\pm 1}) \sigma_i((\Delta^{K}_q)^{\pm 1}) = q^{\lambda_{J,K}} \sigma_i((\Delta^{K}_q)^{\pm 1}) \sigma_i((\Delta^{J}_q)^{\pm 1})$.

The proof is straightforward by checking all cases in (\ref{eq:sigma-map}). Note that the cluster $\x(i)$ is a very special cluster. 

{\bf Case 1}. Both of $\Delta^{J}_q$ and $\Delta^{K}_q$ are frozen variables. According to the formula (\ref{eq:sigma-map}), $\sigma_i$ sends a frozen $\Delta^{J}_q$ either to itself or to $\Delta^{J'}_q (\Delta^{J}_q)^{-1} \Delta^{J''}_q$, where $J'$ is the shift of $J$ to the left by $1$ and $J''$ is the shift of $J$ to the right by $1$.

If both of $\Delta^{J}_q$ and $\Delta^{K}_q$ are not change under $\sigma_i$, it is clear that the quasi-commutative relation is preserved. 

Suppose that under the map $\sigma_i$, $\Delta^{J}_q$ is not changed, and $\Delta^{K}_q$ is changed to $\Delta^{K'}_q (\Delta^{K}_q)^{-1} \Delta^{K''}_q$. We have that $K = [a,b]$, $K'=[a-1, b-1]$, $K''=[a+1,b+1]$ for some $a, b$ (here we only write the case that $a-1\ge 1$, $b+1 \le n$; the cases that $a-1=0 \equiv n \ (\text{mod } n)$ or $b+1=n+1 \equiv 1 \ (\text{mod } n)$ are similar). According to Theorem \ref{thm:scott} and the expression of $K, K', K''$, we have that 
\begin{align*}
& [\Delta^{K'}_q (\Delta^{K}_q)^{-1} \Delta^{K''}_q ]= \Delta^{K'}_q (\Delta^{K}_q)^{-1} \Delta^{K''}_q,
\end{align*}
and
\begin{align*}
& \Delta^{J}_q (\Delta^{K}_q)^{-1} = q^{-\lambda_{J,K}} (\Delta^{K}_q)^{-1} \Delta^{J}_q, \quad \Delta^{J}_q \Delta^{K'}_q = q^{\lambda_{J,K} - 1} \Delta^{K'}_q \Delta^{J}_q, \\
& \Delta^{J}_q \Delta^{K''}_q = q^{\lambda_{J,K} + 1} \Delta^{K''}_q \Delta^{J}_q.
\end{align*}
Therefore 
\begin{align*}
\Delta^{J}_q  [\Delta^{K'}_q (\Delta^{K}_q)^{-1} \Delta^{K''}_q] = q^{\lambda_{J,K}} [\Delta^{K'}_q (\Delta^{K}_q)^{-1} \Delta^{K''}_q]   \Delta^{J}_q.
\end{align*}

For example, in $\CC[\Gr(3,6)]$ we have that $\Delta_q^{123} \Delta_q^{345} = q^2 \Delta_q^{345} \Delta_q^{123}$. Apply $\sigma_2$, we obtain
\begin{align*}
\sigma_2( \Delta_q^{123} ) = \Delta_q^{123}, \quad \sigma_2(\Delta_q^{345}) =[\Delta_q^{234} (\Delta_q^{345})^{-1} \Delta_q^{456}].
\end{align*}
Since
\begin{align*}
\scalemath{0.9}{
\Delta_q^{123} \Delta_q^{234} = q \Delta_q^{234} \Delta_q^{123}, \ \Delta_q^{123} (\Delta_q^{345})^{-1} = q^{-2} (\Delta_q^{345})^{-1} \Delta_q^{123}, \ \Delta_q^{123} \Delta_q^{456} = q^3 \Delta_q^{456} \Delta_q^{123}, }
\end{align*}
we have that $\sigma_2(\Delta_q^{123}) \sigma_2(\Delta_q^{345}) = q^2 \sigma_2(\Delta_q^{345}) \sigma_2(\Delta_q^{123})$.

Suppose that $\sigma_i(\Delta^{J}_q)=[\Delta^{J'}_q (\Delta^{J}_q)^{-1} \Delta^{J''}_q]$ and $\sigma_i(\Delta^{K}_q)=[\Delta^{K'}_q (\Delta^{K}_q)^{-1} \Delta^{K''}_q]$ for some $J'$, $J''$, $K'$, $K''$.
We have that $J = [a,b]$, $J'=[a-1, b-1]$, $J''=[a+1,b+1]$ for some $a, b$ and $K = [c,d]$, $K'=[c-1, d-1]$, $K''=[c+1,d+1]$ for some $c, d$ (here we only write the cases of $a-1>0$, $c-1>0$, $b+1\le n$, $d+1 \le n$; the cases of $a=1$ or $c=1$ or $b=n$ or $d=n$ are similar). 

Using the result in the case above, since $(\Delta^{J}_q)^{-1} \Delta^{K}_q = q^{-\lambda_{J,K}} \Delta^{K}_q (\Delta^{J}_q)^{-1}$, we have that 
\begin{align*}
    (\Delta^{J}_q)^{-1}  [\Delta^{K'}_q (\Delta^{K}_q)^{-1} \Delta^{K''}_q] = q^{-\lambda_{J,K}} [\Delta^{K'}_q (\Delta^{K}_q)^{-1} \Delta^{K''}_q]   (\Delta^{J}_q)^{-1}. 
\end{align*}
We have $\Delta^{J'}_q \Delta^{K}_q = q^{\lambda_{J,K}+1} \Delta^{K}_q \Delta^{J'}_q$. Therefore
\begin{align*}
    \Delta^{J'}_q  [\Delta^{K'}_q (\Delta^{K}_q)^{-1} \Delta^{K''}_q] = q^{\lambda_{J,K}+1} [\Delta^{K'}_q (\Delta^{K}_q)^{-1} \Delta^{K''}_q]   \Delta^{J'}_q. 
\end{align*}
We have $\Delta^{J''}_q \Delta^{K}_q = q^{\lambda_{J,K}-1} \Delta^{K}_q \Delta^{J''}_q$. Therefore
\begin{align*}
    \Delta^{J''}_q  [\Delta^{K'}_q (\Delta^{K}_q)^{-1} \Delta^{K''}_q] = q^{\lambda_{J,K}-1} [\Delta^{K'}_q (\Delta^{K}_q)^{-1} \Delta^{K''}_q]   \Delta^{J''}_q. 
\end{align*}
It follows that 
\begin{align*}
[\Delta^{J'}_q (\Delta^{J}_q)^{-1} \Delta^{J''}_q]  [\Delta^{K'}_q (\Delta^{K}_q)^{-1} \Delta^{K''}_q] = q^{\lambda_{J,K}} [\Delta^{K'}_q (\Delta^{K}_q)^{-1} \Delta^{K''}_q] [\Delta^{J'}_q (\Delta^{J}_q)^{-1} \Delta^{J''}_q].
\end{align*}

{\bf Case 2}. $\Delta^{J}_q$ is a frozen variable and $\Delta^{K}_q$ is a cluster variable. Here we only write the case that $\sigma_i$ does not change $\Delta^{J}_q$, $J$ is an interval, and $\sigma_i$ changes $\Delta^{K}_q$ to $\Delta^{\overline{\sigma}_i(K)}_q$. The other cases are similar. 

By Theorem \ref{thm:scott}, we either have $\lambda_{J,K} = |K_1|-|K_2|$, where $K - J = K_1 \sqcup K_2$, $K_1 \prec J-K \prec K_2$, or $\lambda_{J,K} = |J_2|-|J_1|$, where $J - K = J_1 \sqcup J_2$, $J_1 \prec K-J \prec J_2$.  

Recall that $\overline{\sigma}_i$ is the product $(i,i+1)(d+i, d+i+1)\cdots (n-d+i, n-d+i+1)$ of simple reflections. After applying $\sigma_i$, we have either $\overline{\sigma}_i(K) - J = \overline{\sigma}_i(K_1) \sqcup \overline{\sigma}_i(K_2)$, $\overline{\sigma}_i(K_1) \prec J-\overline{\sigma}_i(K) \prec \overline{\sigma}_i(K_2)$ or $J - \overline{\sigma}_i(K) = J_1 \sqcup J_2$, $J_1 \prec \overline{\sigma}_i(K)-J \prec J_2$. Since $|\overline{\sigma}_i(K_j)|=|K_j|$, $j=1,2$, we have $|\overline{\sigma}_i(K_1)|-|\overline{\sigma}_i(K_2)| = |K_1|-|K_2|$ and hence the quasi-commutative relation does not change.

For example, in $\CC[\Gr(3,6)]$ we have that $\Delta_q^{123} \Delta_q^{146} = q^2 \Delta_q^{146} \Delta_q^{123}$. We have that
\begin{align*}
\sigma_1( \Delta_q^{123} ) = \Delta_q^{123}, \quad \sigma_1(\Delta_q^{146}) = \Delta_q^{256},
\end{align*}
and $\Delta_q^{123} \Delta_q^{256} = q^2 \Delta_q^{256} \Delta_q^{123}$. Therefore $\sigma_1(\Delta_q^{123}) \sigma_1(\Delta_q^{256}) = q^2 \sigma_1(\Delta_q^{256}) \sigma_1(\Delta_q^{123})$.

{\bf Case 3}. Both of $\Delta^{J}_q$ and $\Delta^{K}_q$ are cluster variables. In this case, $\Delta^{J}_q$ is sent to $\Delta^{\overline{\sigma}_i(J)}_q$ and $\Delta^{K}_q$ is sent to $\Delta^{\overline{\sigma}_i(K)}_q$. This case is very similar to Case 2. 

By Theorem \ref{thm:scott}, we either have $\lambda_{J,K} = |K_1|-|K_2|$, where $K - J = K_1 \sqcup K_2$, $K_1 \prec J-K \prec K_2$, or $\lambda_{J,K} = |J_2|-|J_1|$, where $J - K = J_1 \sqcup J_2$, $J_1 \prec K-J \prec J_2$.  

After applying $\sigma_i$, we have either $\overline{\sigma}_i(K) - \overline{\sigma}_i(J) = \overline{\sigma}_i(K_1) \sqcup \overline{\sigma}_i(K_2)$, $\overline{\sigma}_i(K_1) \prec \overline{\sigma}_i(J)-\overline{\sigma}_i(K) \prec \overline{\sigma}_i(K_2)$ or $\overline{\sigma}_i(J) - \overline{\sigma}_i(K) = \overline{\sigma}_i(J_1) \sqcup \overline{\sigma}_i(J_2)$, $\overline{\sigma}_i(J_1) \prec \overline{\sigma}_i(K)-\overline{\sigma}_i(J) \prec \overline{\sigma}_i(J_2)$. Since $|\overline{\sigma}_i(K_j)|=|K_j|$ and $|\overline{\sigma}_i(J_j)|=|J_j|$, $j=1,2$, we have $|\overline{\sigma}_i(K_1)|-|\overline{\sigma}_i(K_2)| = |K_1|-|K_2|$ and $|\overline{\sigma}_i(J_2)|-|\overline{\sigma}_i(J_1)| = |J_2|-|J_1|$, and hence the quasi-commutative relation does not change. 

For example, in $\CC[\Gr(3,6)]$ we have that $\Delta_q^{124} \Delta_q^{134} = q \Delta_q^{134} \Delta_q^{124}$. We have that
\begin{align*}
\sigma_1( \Delta_q^{124} ) = \Delta_q^{125}, \quad \sigma_1(\Delta_q^{134}) = \Delta_q^{235},
\end{align*}
and $\Delta_q^{125} \Delta_q^{235} = q \Delta_q^{235} \Delta_q^{125}$. Therefore $\sigma_1(\Delta_q^{125}) \sigma_1(\Delta_q^{235}) = q \sigma_1(\Delta_q^{235}) \sigma_1(\Delta_q^{125})$.
\end{proof}
 
\begin{lemma} \label{lem:sigma_i on commutative quantum cluster monomials}
In $\CC_q[\Gr(k,n)]$, for a quantum cluster Laurent monomial $[x_1^{a_1}x_2^{a_2}\cdots x_r^{a_r}]$, where $x_1, \ldots, x_r$ are cluster variables or frozen variables which are in the same cluster, $a_i \in \ZZ$, we have that $\sigma_i([x_1^{a_1}x_2^{a_2}\cdots x_r^{a_r}]) = [\sigma_i(x_1)^{a_1} \cdots \sigma_i(x_r)^{a_r}]$ as 
quantum cluster Laurent monomials. 
\end{lemma}

\begin{proof}
By \cite{Fra17} and the definition of $\sigma_i$ in the quantum setting in the beginning of this subsection, we have that $\sigma_i(x_1), \ldots, \sigma_i(x_r)$ are cluster variables in the same cluster (possibly multiply by Laurent monomials of frozen variables). 
Recall that 
\begin{align*}
[x_1^{a_1}x_2^{a_2}\cdots x_r^{a_r}] = q^{ - \frac{1}{2} \sum_{1 \le s <t \le r} a_s a_t  \Lambda( x_s, x_t) } x_1^{ a_1} \cdots  x_r^{ a_r},
\end{align*}
and 
\begin{align*}
[\sigma_i(x_1)^{a_1} \sigma_i(x_2)^{a_2}\cdots \sigma_i(x_r)^{a_r}] = q^{ - \frac{1}{2} \sum_{1 \le s <t \le r} a_s a_t  \Lambda( \sigma_i(x_s), \sigma_i(x_t)) } \sigma_i(x_1)^{ a_1} \cdots  \sigma_i(x_r)^{ a_r}.
\end{align*}
By the definition of $\sigma_i$ in the beginning of this subsection, we have that $\sigma_i(x_1^{ a_1} \cdots  x_r^{ a_r}) = \sigma_i(x_1)^{ a_1} \cdots  \sigma_i(x_r)^{ a_r}$. By Lemma \ref{lem:sigma_i preserves quasicommutative relations}, for any $1 \le s <t \le r$, we have that $\Lambda( x_s, x_t) = \Lambda( \sigma_i(x_s), \sigma_i(x_t))$. Therefore $\sigma_i([x_1^{a_1}x_2^{a_2}\cdots x_r^{a_r}]) = [\sigma_i(x_1)^{a_1} \cdots \sigma_i(x_r)^{a_r}]$. 
\end{proof}

\begin{lemma} \label{lem:sigma_i and varphi commute}
For any cluster monomial $x$ ($x$ can contain frozen variables or inverses of frozen variables) in $\CC[\Gr(k,n)]$ and any $i \in [d-1]$, $d = {\rm gcd}(k,n)$, we have $\varphi(\sigma_i(x))=\sigma_i(\varphi(x))$ and $\varphi(\sigma_i^{-1}(x))=\sigma_i^{-1}(\varphi(x))$. 
\end{lemma}

\begin{proof}
For a cluster variable $x$ in $\CC[\Gr(k,n)]$, by the definition in the beginning of this subsection, we have that 
\begin{align} \label{eq:sigma varphi is varphi sigma on cluster variable}
    \sigma_i(\varphi(x)) = \varphi(\sigma_i(x)).
\end{align}

Suppose that $x = x_1\cdots x_r$, where each $x_i$ is a cluster variable or a frozen variable or the inverse of a frozen variable. By Lemma \ref{lem:sigma_i on commutative quantum cluster monomials}, we have that $\sigma_i(x)=[\sigma_i(x_1)\cdots \sigma_i(x_r)]$. By definition (\ref{equ:extend}), 
\begin{align*}
\varphi(\sigma_i(x)) = \varphi([\sigma_i(x_1)\cdots \sigma_i(x_r)] ) =  [\varphi( \sigma_i(x_1)) \cdots \varphi(\sigma_i(x_r) )] = [\sigma_i(\varphi(x_1)) \cdots \sigma_i(\varphi(x_r))],
\end{align*}
where in the last step we used (\ref{eq:sigma varphi is varphi sigma on cluster variable}). 
Since $[\varphi(x_1)\cdots \varphi(x_r)]$ is a quantum Laurent cluster monomial, by Lemma \ref{lem:sigma_i on commutative quantum cluster monomials} and the definition of $\varphi$, we have that 
\begin{align*}
[\sigma_i(\varphi(x_1)) \cdots \sigma_i(\varphi(x_r))] = \sigma_i([\varphi(x_1)\cdots \varphi(x_r)]) = \sigma_i(\varphi([x_1 \cdots x_r])) = \sigma_i(\varphi(x)).
\end{align*}
 
The prove of $\varphi(\sigma_i^{-1}(x))=\sigma_i^{-1}(\varphi(x))$ is the similar. 
\end{proof}

\begin{remark} \label{remark:sigma of xprime is the same as sigma of expression of mutation formula}
Consider a quantum cluster variable 
\begin{align} \label{eq:xprime obtained from mutation of x in Remark}
x'_r = [x^{-1}_r \prod_{r \to l} x_l] + [x_r^{-1} \prod_{l \to r} x_l],
\end{align}
obtained by mutating $x_r$ in a cluster $\x=(x_1, \ldots, x_m)$. On the one hand, $\sigma_i(x'_r)$ is defined to be $\varphi(\sigma_i(\tilde{x}'_r))$, where $\tilde{x}'_r$ is the cluster variable in $\CC[\Gr(k,n)]$ corresponding to $x'_r$, where the map $\varphi$ is defined in the beginning of this subsection. On the other hand, $\sigma_i(x'_r)$ should also be $\sigma_i([x^{-1}_r \prod_{r \to l} x_l] + [x^{-1}_r \prod_{l \to r} x_l])$. By Lemma \ref{lem:sigma_i on commutative quantum cluster monomials}, $\sigma_i([x^{-1}_r \prod_{r \to l} x_l] + [x_r^{-1} \prod_{l \to r} x_l]) = [\sigma_i(x_r)^{-1} \prod_{r \to l} \sigma_i(x_l)] + [\sigma_i(x_r)^{-1} \prod_{l \to r} \sigma_i(x_l)]$. 

We now verify that $[\sigma_i(x_r)^{-1} \prod_{r \to l} \sigma_i(x_l)] + [\sigma_i(x_r)^{-1} \prod_{l \to r } \sigma_i(x_l)] = \varphi(\sigma_i(\tilde{x}'_r))$. In the beginning of this subsection, we see that the map $\varphi$ sends a cluster in $\CC[\Gr(k,n)]$ to a cluster in $\CC_q[\Gr(k,n)]$. Let $\tilde{x}_j$ be the cluster variable in $\CC[\Gr(k,n)]$ such that $\varphi(\tilde{x}_j)=x_j$. The exchange relation in $\CC[\Gr(k,n)]$ corresponding to (\ref{eq:xprime obtained from mutation of x in Remark}) is $\tilde{x}'_r = \tilde{x}^{-1}_r \prod_{r  \to l} \tilde{x}_l + \tilde{x}_r^{-1} \prod_{l \to r } \tilde{x}_l$. We have that
\begin{align*}
\sigma_i( \tilde{x}'_r ) = \sigma_i(\tilde{x}_r)^{-1} \prod_{r  \to l} \sigma_i(\tilde{x}_l) + \sigma_i(\tilde{x}_r)^{-1} \prod_{l \to r } \sigma_i(\tilde{x}_l).
\end{align*}
This is an exchange relation in $\CC[\Gr(k,n)]$ up to multiplying frozen variables. As discussed in the beginning of this subsection, if we replace $\sigma_i(\tilde{x}_r)$ by $\varphi(\sigma_i(\tilde{x}_r))$, then we obtain an exchange relation in $\CC_q[\Gr(k,n)]$ up to multiplying frozen variables. That is, $\varphi(\sigma_i( \tilde{x}'_r ))$ is equal to 
\begin{align*}
    [\varphi(\sigma_i(\tilde{x}_r))^{-1} \prod_{r  \to l} \varphi(\sigma_i(\tilde{x}_l))] + [\varphi(\sigma_i(\tilde{x}_r))^{-1} \prod_{l \to r } \varphi(\sigma_i(\tilde{x}_l))].
\end{align*}
By Lemma \ref{lem:sigma_i and varphi commute}, this is equal to 
\begin{align*}
    [\sigma_i(\varphi(\tilde{x}_r))^{-1} \prod_{r  \to l} \sigma_i(\varphi(\tilde{x}_l))] + [\sigma_i(\varphi(\tilde{x}_r))^{-1} \prod_{l \to r } \sigma_i(\varphi(\tilde{x}_l))],
\end{align*}
which is equal to $[\sigma_i(x_r)^{-1} \prod_{r  \to l} \sigma_i(x_l)] + [\sigma_i(x_r)^{-1} \prod_{l \to r } \sigma_i(x_l)]$.

\end{remark}

\begin{lemma} \label{lem:sigma_i preserves exchange relations}
For $i \in [d-1]$, $\sigma_i$ and $\sigma_i^{-1}$ preserve exchange relations in $\CC_q[\Gr(k,n)]$.
\end{lemma}

\begin{proof}
We will prove the result for $\sigma_i$. The proof of the result for $\sigma_i^{-1}$ is similar. 

It is shown in \cite{GL11, GL14} that $\CC_q[\Gr(k,n)]$ has a quantum cluster algebra structure. Let $x$ be a quantum cluster variable in a cluster. Mutate at $x$ we obtain that 
\begin{align} \label{eq:mutation at x obtain x prime}
x' = [x^{-1}\prod_{x_j \to x} x_j] + [x^{-1}\prod_{x \to x_j} x_j].
\end{align} 
We rewrite the identity using non-commutative product:
\begin{align} \label{eq:mutation at x obtain xprime noncommutative product}
x' = q^a x^{-1} c_1 \cdots  c_r + q^b x^{-1} d_1  \cdots  d_s,
\end{align}
where $c_1, \ldots, c_r, d_1, \ldots, d_s$ are $x_j$'s in (\ref{eq:mutation at x obtain x prime}) in some order, and $a, b \in \frac{1}{2}\ZZ$. Therefore
\begin{align*}
x x' = q^a c_1 \cdots  c_r + q^b d_1  \cdots  d_s.
\end{align*}
This is an exchange relation in $\CC_q[\Gr(k,n)]$. We now verify that 
\begin{align}  \label{eq:sigmax times sigmaxprime}
\sigma_i(x) \sigma_i(x') = q^a \sigma_i(c_1) \cdots  \sigma_i(c_r) + q^b \sigma_i(d_1)  \cdots  \sigma_i(d_s).
\end{align}

By Lemma \ref{lem:sigma_i on commutative quantum cluster monomials} and Remark \ref{remark:sigma of xprime is the same as sigma of expression of mutation formula}, apply $\sigma_i$ to (\ref{eq:mutation at x obtain x prime}), we have that 
\begin{align} \label{eq:sigmaxprime}
\sigma_i(x') = [\sigma_i(x)^{-1} \prod_{x_j \to x} \sigma_i(x_j)] + [\sigma_i(x)^{-1} \prod_{x \to x_j} \sigma_i(x_j)].
\end{align} 

Note that $c_1, \ldots, c_r, d_1, \ldots, d_s$ are $x_j$'s in (\ref{eq:mutation at x obtain x prime}) in some order. There are some $a', b'$ such that the right hand side of (\ref{eq:sigmaxprime}) is 
\begin{align*}
q^{a'} \sigma_i(x)^{-1} \sigma_i(c_1) \cdots  \sigma_i(c_r) + q^{b'} \sigma_i(x)^{-1} \sigma_i(d_1)  \cdots  \sigma_i(d_s).
\end{align*}
By Lemma \ref{lem:sigma_i preserves quasicommutative relations}, $c_j c_{j'} = q^{\lambda_{j,j'}}c_j c_{j'}$ implies that $\sigma_i(c_j) \sigma_i(c_{j'}) = q^{\lambda_{j,j'}}\sigma_i(c_j) \sigma_i(c_{j'})$. Therefore (\ref{eq:mutation at x obtain x prime}), (\ref{eq:mutation at x obtain xprime noncommutative product}), and (\ref{eq:sigmaxprime}) implies that $a=a'$, $b=b'$. Hence
\begin{align*}
\sigma_i(x') = q^{a} \sigma_i(x)^{-1} \sigma_i(c_1) \cdots  \sigma_i(c_r) + q^{b} \sigma_i(x)^{-1} \sigma_i(d_1)  \cdots  \sigma_i(d_s).
\end{align*}
Multiply $\sigma_i(x)$ from the left hand side, we obtain (\ref{eq:sigmax times sigmaxprime}).

\end{proof}

The $3$-term quantum Pl\"{u}cker relations are special exchange relations. To the best of our knowledge, it is still not known whether or not the quantum Grassmannian $\CC_q[\Gr(k,n)]$ has a presentation given by quantum Pl\"{u}cker coordinates modulo all quantum Pl\"{u}cker relations. It might be possible that $\CC_q[\Gr(k,n)]$ has other relations which cannot be derived from quantum Pl\"{u}cker relations. We conjecture the following.

\begin{conjecture} \label{conj:sigmai preserves r term plucker relations r is 4 or more}
The maps $\sigma_i$ and $\sigma_i^{-1}$ also preserve $r$-term quantum Pl\"{u}cker relations for $r \ge 4$ in $\CC_q[\Gr(k,n)]$ and other relations in $\CC_q[\Gr(k,n)]$ which cannot be derived from quantum Pl\"{u}cker relations (if any).
\end{conjecture}

\begin{remark}
Let $\Gr^{\circ}(k,n) \subset \Gr(k,n)$ be the quasi-affine variety defined by non-vanishing of the cyclically consecutive Pl\"{u}cker coordinates $P_{i+1, \ldots, i+k}$, $i \in [n]$, and the indices are treated modulo $n$. The coordinate ring $\CC[\Gr^{\circ}(k,n)]$ is the localization of $\CC[\Gr(k,n)]$ at the cyclically consecutive Pl\"{u}cker coordinates. The distinction between $\CC[\Gr^{\circ}(k,n)]$ and $\CC[\Gr(k,n)]$ is unimportant from a combinatorial perspective, and in \cite{Fra17} Fraser freely translate between them, see Section 3 in \cite{Fra17}. In \cite{Fra17}, Fraser proved that $\sigma_i$'s are regular automorphisms of $\Gr^{\circ}(k,n)$. These automorphisms induce algebra automorphisms on $\CC[\Gr^{\circ}(k,n)]$. 

In \cite{FWZ}, Section 6.8, Example 6.8.6, it is observed that in general, the ideal of relations among the cluster variables of $\CC[\Gr(k,n)]$ is not generated by the exchange relations. Therefore in quantum case, in general, the ideal of relations among the cluster variables of $\CC_q[\Gr(k,n)]$ is not generated by quasi-commutative relations and the exchange relations. 

In this paper, we only proved that $\sigma_i$'s preserve quasi-commutative relations and exchange relations, and $\sigma_i$'s are automorphisms on $\CC_q[\Gr(k,n)]$.
\end{remark}

We give an example that $\sigma_i$ preserves a $4$-term quantum Pl\"{u}cker relation to show evidence of Conjecture \ref{conj:sigmai preserves r term plucker relations r is 4 or more}.
\begin{example}
In $\CC_q[\Gr(3,6)]$, we have that 
\begin{align*} 
\Delta_q^{124} \Delta_{q}^{356} = \frac{1}{q} \Delta_q^{123} \Delta_q^{456} + q \Delta_q^{125} \Delta_q^{346} - q^2 \Delta_q^{126} \Delta_q^{345}.
\end{align*}
We now verify that $\sigma_1$ preserves the above identity. That is, by Table \ref{table:braid group action on Gr36}, we need to verify that
\begin{align*} 
\Delta_q^{125} [\Delta_{q}^{136}(\Delta_q^{156})^{-1}\Delta_q^{456} ] = \frac{1}{q} \Delta_q^{123} \Delta_q^{456} + q [\Delta_q^{126}(\Delta_q^{156})^{-1}\Delta_q^{145}] \Delta_q^{356} - q^2 \Delta_q^{126} \Delta_q^{345}.
\end{align*}
We have that 
\begin{align*}
& [\Delta_{q}^{136}(\Delta_q^{156})^{-1}\Delta_q^{456}] = q^{-1} \Delta_{q}^{136}   \Delta_q^{456}   (\Delta_q^{156})^{-1}, \\
& [\Delta_q^{126}(\Delta_q^{156})^{-1}\Delta_q^{145}] = q \Delta_q^{126} \Delta_q^{145} (\Delta_q^{156})^{-1}.
\end{align*}
Therefore we need to verify
\begin{align*} 
\scalemath{0.9}{
q^{-1} \Delta_q^{125} \Delta_{q}^{136}   \Delta_q^{456}   (\Delta_q^{156})^{-1} = \frac{1}{q} \Delta_q^{123} \Delta_q^{456} + q^2 \Delta_q^{126} \Delta_q^{145} (\Delta_q^{156})^{-1} \Delta_q^{356} - q^2 \Delta_q^{126} \Delta_q^{345}. }
\end{align*}
We have that $(\Delta_q^{156})^{-1} \Delta_q^{356} = q^{-1} \Delta_q^{356} (\Delta_q^{156})^{-1}$. Therefore we need to verify that
\begin{align*} 
\scalemath{0.9}{
q^{-1} \Delta_q^{125} \Delta_{q}^{136}   \Delta_q^{456}   (\Delta_q^{156})^{-1} = \frac{1}{q} \Delta_q^{123} \Delta_q^{456} + q \Delta_q^{126} \Delta_q^{145} \Delta_q^{356} (\Delta_q^{156})^{-1} - q^2 \Delta_q^{126} \Delta_q^{345}. }
\end{align*}
That is,
\begin{align*} 
\scalemath{0.9}{
 \Delta_q^{125} \Delta_{q}^{136}   \Delta_q^{456} =   \Delta_q^{123} \Delta_q^{456} \Delta_q^{156} + q^2 \Delta_q^{126} \Delta_q^{145} \Delta_q^{356} - q^3 \Delta_q^{126} \Delta_q^{345} \Delta_q^{156}. }
\end{align*}
This is checked directly on computer. 
\end{example}

\begin{lemma} \label{lem:sigmai and sigmai inverse are algebra automorphisms}
Up to Conjecture \ref{conj:sigmai preserves r term plucker relations r is 4 or more}, the maps $\sigma_i$ and $\sigma_i^{-1}$ are $\CC_q$-algebra automorphisms of the quantum cluster algebra $\CC_q[\Gr(k,n)]$.
\end{lemma}

\begin{proof}
Assume that Conjecture \ref{conj:sigmai preserves r term plucker relations r is 4 or more} is true. Then lemmas \ref{lem:sigma_i preserves quasicommutative relations} and \ref{lem:sigma_i preserves exchange relations} imply that $\sigma_i$ and $\sigma_i^{-1}$ are algebra homomorphisms of $\CC_q[\Gr(k,n)]$.

We now check that $\sigma_i$ and $\sigma_i^{-1}$ are inverses of each other on $\CC_q[\Gr(k,n)]$. In the beginning of this subsection, we see that every quantum Pl\"{u}cker coordinate of $\CC_q[\Gr(k,n)]$ is of the form $\varphi(x)$, where $x$ is a Pl\"{u}cker coordinate of $\CC[\Gr(k,n)]$. By Lemma \ref{lem:sigma_i and varphi commute}, we have that $\sigma_i^{-1}(\sigma_i(\varphi(x))) = \sigma_i^{-1}(\varphi(\sigma_i(x))) = \varphi(\sigma_i^{-1}(\sigma_i(x)))$. It is shown in \cite{Fra17} that $\sigma_i$, $\sigma_i^{-1}$ are inverses of each other on $\CC[\Gr(k,n)]$. Therefore $\varphi(\sigma_i^{-1}(\sigma_i(x)))=\varphi(x)$ and hence $\sigma_i^{-1}(\sigma_i(\varphi(x)))=\varphi(x)$. Similarly, $\sigma_i(\sigma_i^{-1}(\varphi(x)))=\varphi(x)$. 

Every element in $\CC_q[\Gr(k,n)]$ is a $\CC_q$-linear combination of products of quantum Pl\"{u}cker coordinates of $\CC_q[\Gr(k,n)]$. To show that $\sigma_i$ and $\sigma_i^{-1}$ are inverses of each other on $\CC_q[\Gr(k,n)]$, it suffices to shown that $\sigma_i^{-1}(\sigma_i(x_1 \cdots  x_r))=x_1 \cdots  x_r$ and $\sigma_i(\sigma_i^{-1}(x_1 \cdots  x_r))=x_1 \cdots  x_r$, where $x_i$'s are quantum Pl\"{u}cker coordinates. This follows from the fact that $\sigma_i$, $\sigma_i^{-1}$ are algebra homomorphisms on $\CC_q[\Gr(k,n)]$ and $\sigma_i^{-1}(\sigma_i(x_i))=x_i$, $\sigma_i(\sigma_i^{-1}(x_i))=x_i$ for quantum Pl\"{u}cker coordinates $x_i$.
\end{proof}

Our main result of this section is the following.
\begin{theorem}\label{thm: braids are QQH}
Let $k \le n$ and $d={\rm gcd}(k,n)$.  
\begin{enumerate}
\item[(i)] Assume that Conjecture \ref{conj:sigmai preserves r term plucker relations r is 4 or more} is ture. Then for any $1 \leq i \leq d-1$, the $\sigma_i$ is a quasi-automorphism of $\CC_q[\Gr(k,n)]$.

\item[(ii)] The maps $\sigma_i$'s give rise to a braid group action on $\CC_q[\Gr(k,n)]$, that is, they satisfy relations
\begin{align*}
& \sigma_i \sigma_j = \sigma_j \sigma_i, \quad |i-j|>1, \\
& \sigma_i \sigma_j \sigma_i = \sigma_j \sigma_i \sigma_j, \quad |i-j|=1.
\end{align*} 
\end{enumerate}
\end{theorem}

\begin{remark}
In the case that $k=2$, the quantum Pl\"{u}cker relations in $\CC_q[{\rm Gr}(2,n)]$ are $3$-term relations, see \cite{GL11}. Therefore we have that Theorem \ref{thm: braids are QQH} is true without the assumption that Conjecture \ref{conj:sigmai preserves r term plucker relations r is 4 or more} is true.
\end{remark}

We first give some examples in the next two subsections and we prove Theorem \ref{thm: braids are QQH} in Section \ref{sec:prove main theorem}. 

\subsection{Braid group action on $\CC_q[\Gr(2,4)]$} \label{sec:Gr24}
By definition, $\sigma_1$ sends
\begin{align*}
& \Delta_{q}^{12} \mapsto \Delta_{q}^{12}, \ \Delta_{q}^{13} \mapsto \Delta_{q}^{24}, \ \Delta_{q}^{14} \mapsto \Delta_{q}^{12} (\Delta_{q}^{14})^{-1} \Delta_{q}^{34}, \ \Delta_{q}^{23} \mapsto \Delta_{q}^{12} (\Delta_{q}^{23})^{-1} \Delta_{q}^{34}, \ \Delta_{q}^{34} \mapsto \Delta_{q}^{34}.
\end{align*}

We have that
\begin{align*}
\Delta_q^{24} & = (\Delta_q^{13})' = [(\Delta_q^{13})^{-1} \Delta_q^{12} \Delta_q^{34}] + [(\Delta_q^{13})^{-1} \Delta_q^{14} \Delta_q^{23}] \\
& = \frac{1}{q} (\Delta_q^{13})^{-1}   \Delta_q^{12}   \Delta_q^{34} + q (\Delta_q^{13})^{-1}   \Delta_q^{14}   \Delta_q^{23}.
\end{align*}

By Lemma \ref{lem:sigma_i on commutative quantum cluster monomials},
\begin{align*}
\sigma_1(\Delta_q^{24}) & = [\sigma_1((\Delta_q^{13})^{-1}) \sigma_1(\Delta_q^{12}) \sigma_1(\Delta_q^{34})] + [\sigma_1((\Delta_q^{13})^{-1}) \sigma_1(\Delta_q^{14}) \sigma_1(\Delta_q^{23})] \\
& = [(\Delta_q^{24})^{-1} \Delta_q^{12} \Delta_q^{34}] + [(\Delta_q^{24})^{-1} \Delta_q^{12} (\Delta_q^{14})^{-1} \Delta_q^{34} \Delta_q^{12} (\Delta_q^{23})^{-1} \Delta_q^{34}] \\
& = q^{-1} (\Delta_q^{24})^{-1}   \Delta_q^{12}   \Delta_q^{34} + q (\Delta_q^{24})^{-1}   \Delta_q^{12}   ( \Delta_q^{14})^{-1}   \Delta_q^{34}   \Delta_q^{12}   ( \Delta_q^{23})^{-1}   \Delta_q^{34}.
\end{align*}

We now verify that $\sigma_1$ preserve the Pl\"{u}cker relation:
\begin{align*}
\Delta_q^{13}   \Delta_q^{24} = q^{-1} \Delta_q^{12}   \Delta_q^{34} + q  \Delta_q^{14}   \Delta_q^{23}.
\end{align*}

On the left hand side, $\sigma_1(\Delta_q^{13}) \sigma_1(\Delta_q^{24})$ equals to
\begin{align*}
& \Delta_q^{24}   (q^{-1} (\Delta_q^{24})^{-1}   \Delta_q^{12}   \Delta_q^{34} + q (\Delta_q^{24})^{-1}   \Delta_q^{12}   ( \Delta_q^{14})^{-1}   \Delta_q^{34}   \Delta_q^{12}   ( \Delta_q^{23})^{-1}   \Delta_q^{34}) \\
& = q^{-1} \Delta_q^{12}   \Delta_q^{34} + q \Delta_q^{12}   ( \Delta_q^{14})^{-1}   \Delta_q^{34}   \Delta_q^{12}   ( \Delta_q^{23})^{-1}   \Delta_q^{34}.
\end{align*}
On the right hand side, $q^{-1} \sigma_1( \Delta_q^{12} )   \sigma_1( \Delta_q^{34} ) + q  \sigma_1( \Delta_q^{14} )   \sigma_1( \Delta_q^{23} )$ equals to
\begin{align*}
& q^{-1} \Delta_q^{12}    \Delta_q^{34} + q  [\Delta_{q}^{12} (\Delta_{q}^{14})^{-1} \Delta_{q}^{34}]   [\Delta_{q}^{12} (\Delta_{q}^{23})^{-1} \Delta_{q}^{34}] \\
& = q^{-1} \Delta_q^{12}    \Delta_q^{34} + q  \Delta_{q}^{12}   (\Delta_{q}^{14})^{-1}   \Delta_{q}^{34}   \Delta_{q}^{12}   (\Delta_{q}^{23})^{-1}   \Delta_{q}^{34}.
\end{align*}

Therefore
\begin{align*}
\sigma_1( \Delta_q^{13} )   \sigma_1(\Delta_q^{24}) = q^{-1} \sigma_1( \Delta_q^{12} )   \sigma_1( \Delta_q^{34} ) + q  \sigma_1( \Delta_q^{14} )   \sigma_1( \Delta_q^{23} ).
\end{align*}

We now check that the relation $\Delta_q^{14}   \Delta_q^{23} = \Delta_q^{23}   \Delta_q^{14}$ is preserved by $\sigma_1$. On the left hand side,
\begin{align*}
\sigma_1(\Delta_q^{14}) \sigma_1(\Delta_q^{23}) & = [\Delta_{q}^{12} (\Delta_{q}^{14})^{-1} \Delta_{q}^{34} ]   [\Delta_{q}^{12} (\Delta_{q}^{23})^{-1} \Delta_{q}^{34} ] \\
& = \Delta_{q}^{12}   (\Delta_{q}^{14})^{-1}   \Delta_{q}^{34}   \Delta_{q}^{12}   (\Delta_{q}^{23})^{-1}   \Delta_{q}^{34} \\
& =  [\Delta_{q}^{12} (\Delta_{q}^{14})^{-1} \Delta_{q}^{34} \Delta_{q}^{12} (\Delta_{q}^{23})^{-1} \Delta_{q}^{34}].
\end{align*}
On the right hand side,
\begin{align*}
\sigma_1(\Delta_q^{23}) \sigma_1(\Delta_q^{14}) & = [\Delta_{q}^{12} (\Delta_{q}^{23})^{-1} \Delta_{q}^{34}]    [\Delta_{q}^{12} (\Delta_{q}^{14})^{-1} \Delta_{q}^{34} ]  \\
& = \Delta_{q}^{12}   (\Delta_{q}^{23})^{-1}   \Delta_{q}^{34}   \Delta_{q}^{12}   (\Delta_{q}^{14})^{-1}   \Delta_{q}^{34} \\
& =  [\Delta_{q}^{12} (\Delta_{q}^{23})^{-1} \Delta_{q}^{34} \Delta_{q}^{12} (\Delta_{q}^{14})^{-1} \Delta_{q}^{34}].
\end{align*}
Therefore $\sigma_1(\Delta_q^{14}) \sigma_1(\Delta_q^{23}) = \sigma_1(\Delta_q^{23}) \sigma_1(\Delta_q^{14})$.

\begin{remark}
It is shown in \cite[Section 3]{LL11} that the cycling map is not an automorphism of the quantum Grassmannian $\CC_q[\Gr(k,n)]$.

In $\CC_q[\Gr(2,4)]$, the map $\sigma_1$ is similar but different from the cycling map $\rho$. In $\CC_q[\Gr(2,4)]$, the cycling map is given by $\Delta_q^{12} \mapsto \Delta_q^{23}$, $\Delta_q^{13} \mapsto \Delta_q^{24}$, $\Delta_q^{23} \mapsto \Delta_q^{34}$, $\Delta_q^{24} \mapsto \Delta_q^{13}$, $\Delta_q^{34} \mapsto \Delta_q^{14}$, $\Delta_q^{14} \mapsto \Delta_q^{12}$.
We have that $\Delta_q^{23}   \Delta_q^{14} = \Delta_q^{14}   \Delta_q^{23}$.
Applying the cycling map to $\Delta_q^{23}, \Delta_q^{14}$, we obtain $\Delta_q^{34}, \Delta_q^{12}$ respectively. Although $\Delta_q^{23}, \Delta_q^{14}$ commute, $\Delta_q^{34}, \Delta_q^{12}$ do not commute and we have  $\Delta_q^{12}   \Delta_q^{34} = q^2 \Delta_q^{34}   \Delta_q^{12}$. No matter how we choose $a_{iji'j'}$ in the deformed cycling map $\Delta_q^{ij} \mapsto q^{a_{iji'j'}} \Delta_q^{i'j'}$, the deformed cycling map always sends the relation $\Delta_q^{23}   \Delta_q^{14} = \Delta_q^{14}   \Delta_q^{23}$ to the relation $\Delta_q^{12}   \Delta_q^{34} =  \Delta_q^{34}   \Delta_q^{12}$ which does not hold in $\CC_q[\Gr(2,4)]$.
\end{remark}

\subsection{Braid group action on $\CC_q[\Gr(3,6)]$} \label{sec:Gr36}

\begin{center}
\begin{table}
\resizebox{0.99\textwidth}{!}
{
\begin{tabular}{|c|c|c|}
\hline
variable $x$ & $\sigma_1(x)$ & $\sigma_2(x)$ \\
\hline
 $\Delta_q^{123}$ & $\Delta_q^{123}$   & $\Delta_q^{123}$ \\
 \hline
 $\Delta_q^{234}$ & $[\Delta_q^{123} (\Delta_q^{234})^{-1} \Delta_q^{345}]$    & $\Delta_q^{234}$ \\
 \hline
 $\Delta_q^{345}$ & $\Delta_q^{345}$  & $[\Delta_q^{234}(\Delta_q^{345})^{-1} \Delta_q^{456}]$ \\
 \hline
 $\Delta_q^{456}$ & $\Delta_q^{456}$  & $\Delta_q^{456}$ \\
\hline
 $\Delta_q^{126}$ & $\Delta_q^{126}$ & $[\Delta_q^{123} (\Delta_q^{126})^{-1} \Delta_q^{156}]$  \\
 \hline
 $\Delta_q^{156}$ & $[\Delta_q^{126} (\Delta_q^{156})^{-1} \Delta_q^{456}]$  & $\Delta_q^{156}$   \\
\hline
 $\Delta_q^{124}$ & $\Delta_q^{125}$  & $\Delta_q^{134}$ \\
\hline
 $\Delta_q^{125}$ & $[\Delta_q^{126} ( \Delta_q^{156} )^{-1} \Delta_q^{145}]$  & $\Delta_q^{136}$ \\
\hline
 $\Delta_q^{134}$ & $\Delta_q^{235}$ & $[\Delta_q^{145} (\Delta_q^{345})^{-1} \Delta_q^{234}]$  \\
\hline
 $\Delta_q^{135}$ & $[(\Delta_q^{156})^{-1}z]$ & $[(\Delta_q^{345})^{-1}z]$   \\
\hline
 $\Delta_q^{136}$ & $\Delta_q^{236}$ & $[\Delta_q^{123}\Delta_q^{156}\Delta_q^{245}(\Delta_q^{126})^{-1}(\Delta_q^{345})^{-1}]$ \\
\hline
 $\Delta_q^{145}$ & $\Delta_q^{245}$ &  $\Delta_q^{146}$  \\
\hline
$\Delta_q^{146}$ & $\Delta_q^{256}$ &  $[\Delta_q^{124}(\Delta_q^{126})^{-1}\Delta_q^{156}]$  \\
\hline
$\Delta_q^{235}$ & $[\Delta_q^{123}\Delta_q^{146}(\Delta_q^{156})^{-1}(\Delta_q^{234})^{-1}\Delta_q^{345}]$ & $\Delta_q^{236}$  \\
\hline
 $\Delta_q^{236}$ & $[\Delta_q^{123}(\Delta_q^{234})^{-1}\Delta_q^{346}]$ & $[\Delta_q^{123}(\Delta_q^{126})^{-1}\Delta_q^{256}]$  \\
\hline
$\Delta_q^{245}$ & $[\Delta_q^{124} (\Delta_q^{234})^{-1} \Delta_q^{345}]$ &  $\Delta_q^{346}$  \\
\hline
 $\Delta_q^{246}$ & $[(\Delta_q^{234})^{-1}y]$ & $[(\Delta_q^{126})^{-1}y]$  \\
\hline
 $\Delta_q^{256}$ & $[\Delta_q^{126} \Delta_q^{134} \Delta_q^{456} (\Delta_q^{156})^{-1} (\Delta_q^{234})^{-1}] $  & $\Delta_q^{356}$ \\
\hline
 $\Delta_q^{346}$ & $\Delta_q^{356}$ & $[\Delta_q^{125}(\Delta_q^{126})^{-1}\Delta_q^{234}(\Delta_q^{345})^{-1}\Delta_q^{456}]$  \\
\hline
 $\Delta_q^{356}$ & $[\Delta_q^{136}(\Delta_q^{156})^{-1}\Delta_q^{456}]$ & $[\Delta_q^{235} (\Delta_q^{345})^{-1} \Delta_q^{456}]$  \\
\hline
$y$ & $[\Delta_q^{126} \Delta_q^{135} (\Delta_q^{156})^{-1} \Delta_q^{456}] $ &    $[\Delta_q^{135} \Delta_q^{234} (\Delta_q^{345})^{-1} \Delta_q^{456}]$ \\
\hline
$z$ & $[\Delta_q^{123}(\Delta_q^{234})^{-1}\Delta_q^{246}\Delta_q^{345}]$ & $[\Delta_q^{123} (\Delta_q^{126})^{-1} \Delta_q^{156} \Delta_q^{246}]$ \\
\hline
\end{tabular}
}
\caption{Braid group action on quantum cluster variables in $\mathbb{C}_q[\Gr(3,6)]$. 
Here $y=q^{-\frac{3}{2}}( \Delta_q^{124} \Delta_q^{356} - \frac{1}{q} \Delta_q^{123} \Delta_q^{456} )$ and $z=q^{-\frac{1}{2}}( \Delta_q^{145} \Delta_q^{236} - \frac{1}{q^2} \Delta_q^{123} \Delta_q^{456} )$.}
\label{table:braid group action on Gr36}
\end{table}
\end{center}

We list the images of quantum cluster variables of $\CC_q[\Gr(3,6)]$ under the maps $\sigma_i$, $i=1,2$, in Table \ref{table:braid group action on Gr36}.

We now check that quantum Pl\"{u}cker relations are preserved by $\sigma_1, \sigma_2$. We check one example and others are similar. Consider the quantum Pl\"{u}cker relation
\begin{align} \label{eq:plucker 124 136}
\Delta_q^{124}   \Delta_q^{136} = \frac{1}{q} \Delta_q^{123}   \Delta_q^{146} + q \Delta_q^{126}   \Delta_q^{134}.
\end{align}

We have that
\begin{align*}
\Delta_q^{136} & = [(\Delta_q^{124})^{-1} \Delta_q^{123} \Delta_q^{146}] + [(\Delta_q^{124})^{-1} \Delta_q^{126} \Delta_q^{134}],
\end{align*}
and
\begin{align*}
\sigma_1( \Delta_q^{136} ) & = [\sigma_1( ( \Delta_q^{124})^{-1} ) \sigma_1( \Delta_q^{123} ) \sigma_1( \Delta_q^{146})] + [\sigma_1( ( \Delta_q^{124})^{-1} ) \sigma_1( \Delta_q^{126} ) \sigma_1( \Delta_q^{134} )] \\
& =  [(\Delta_q^{125})^{-1} \Delta_q^{123} \Delta_q^{256}] + [(\Delta_q^{125})^{-1} \Delta_q^{126} \Delta_q^{235}] \\
& = \frac{1}{q} (\Delta_q^{125})^{-1}   \Delta_q^{123}   \Delta_q^{256} + q (\Delta_q^{125})^{-1}   \Delta_q^{126}   \Delta_q^{235}.
\end{align*}
On the left hand side,
\begin{align*}
\sigma_1( \Delta_q^{124})   \sigma_1( \Delta_q^{136} ) & = \Delta_q^{125}   (\frac{1}{q} (\Delta_q^{125})^{-1}   \Delta_q^{123}   \Delta_q^{256} + q (\Delta_q^{125})^{-1}   \Delta_q^{126}   \Delta_q^{235}) \\
& = \frac{1}{q}  \Delta_q^{123}   \Delta_q^{256} + q  \Delta_q^{126}   \Delta_q^{235}.
\end{align*}
On the right hand side,
\begin{align*}
\frac{1}{q} \sigma_1( \Delta_q^{123})   \sigma_1( \Delta_q^{146}) + q \sigma_1( \Delta_q^{126} )   \sigma_1( \Delta_q^{134} ) = \frac{1}{q} \Delta_q^{123}   \Delta_q^{256} + q \Delta_q^{126}   \Delta_q^{235}.
\end{align*}
Therefore $\sigma_1$ preserves the Pl\"{u}cker relation (\ref{eq:plucker 124 136}).

We now verify that $\sigma_1$ and $\sigma_2$ satisfy the braid relation
$\sigma_1\sigma_2\sigma_1=\sigma_2\sigma_1\sigma_2$. Consider $\sigma_1\sigma_2\sigma_1(\Delta_q^{145})$ and $\sigma_2\sigma_1\sigma_2(\Delta_q^{145})$. We have that
\begin{align*}
& \sigma_1\sigma_2\sigma_1(\Delta_q^{145})=\sigma_1\sigma_2(\Delta_q^{245})=\sigma_1(\Delta_q^{346})=\Delta_q^{356}, 
\\
& 
\sigma_2\sigma_1\sigma_2(\Delta_q^{145})=\sigma_2\sigma_1(\Delta_q^{146})=\sigma_2(\Delta_q^{256})=\Delta_q^{356}.
\end{align*}
Similarly, we have $\sigma_1\sigma_2\sigma_1(x_i)=\sigma_2\sigma_1\sigma_2(x_i)$ for other quantum cluster variables $x_i$ in the initial cluster.

\section{Proof of Theorem \ref{thm: braids are QQH}} \label{sec:prove main theorem}

In this section, we prove our main result Theorem \ref{thm: braids are QQH}.

\subsection{Proof of (i) of Theorem \ref{thm: braids are QQH}} \label{subsec:proof (i) of theorem braids are quasi-homomorphism}

By Lemma \ref{lem:sigmai and sigmai inverse are algebra automorphisms}, the maps $\sigma_i$ and $\sigma_i^{-1}$ are algebra automorphisms of the quantum cluster algebra $\CC_q[\Gr(k,n)]$.
 
By the definition of $\sigma_i$, the matrix associated to $\sigma_i$ is $R=\left(
      \begin{array}{cc}
        I_{n\times n} & 0 \\
        0 & L \\
      \end{array}
    \right)$, where $L=(l_{jk})$ is an $n\times n$ matrix with entries
\begin{equation}
\label{eq:M'}
l_{jk} =
\begin{cases}
1,  & \text{if $j=l \not \equiv i+1, j+1=l \equiv i+1, ~or~ j-1=l \equiv i+1,$} \\
-1, & \text{if $j=l \equiv i+1,$} \\
0, & \text{otherwise.}
\end{cases}
\end{equation}
It is not hard to see that $L^{-1}=L$ and $R^{-1}=R$. Here we place the frozen variables, i.e., consecutive quantum Pl\"{u}cker coordinates, in the order of $[1,k], [2,k+1],\ldots,[n-k+1,n],[n-k+2,n+1],\ldots, [n,n+k-1]$, where the addition is taken in $\mathbb Z_n$.
 
It is clear that by the definition, the maps $\sigma_i$ and $\sigma_i^{-1}$ satisfies the condition $(2)$ in Proposition \ref{prop:equ-quasi-homo}. On the other hand, the conditions $(1)$ and $(3)$ in Proposition \ref{prop:equ-quasi-homo} follow from \cite[Lemma 5.11]{Fra17}. So to prove Theorem \ref{thm: braids are QQH} (i) we only need to check the condition (4) in Proposition \ref{prop:equ-quasi-homo}, that is, $R^t\overline{\Lambda}R=\Lambda$. This follows from 
Lemma \ref{lem:sigma_i preserves quasicommutative relations}. 

\subsection{Proof of (ii) of Theorem \ref{thm: braids are QQH}} \label{subsec:proof (ii) of theorem braids are quasi-homomorphism}

Recall that in Section \ref{subsec:the bijection varphi between cluster monomials and quantum cluster monomials}, we defined a map $\varphi$ sending a cluster Laurent monomial in $\CC[\Gr(k,n)]$ to a quantum cluster Laurent monomial in $\CC_q[\Gr(k,n)]$. Every quantum cluster variable in $\CC_q[\Gr(k,n)]$ is of the form $\varphi(x)$, where $x$ is a cluster variable in $\CC[\Gr(k,n)]$. By Lemma \ref{lem:sigma_i and varphi commute}, we have that for every $i$ and every cluster Laurent monomial $x$ ($x$ can contain frozen variables or inverses of frozen variables) in $\CC[\Gr(k,n)]$, $\sigma_i(\varphi(x))=\varphi(\sigma_i(x))$. In particular, for any Pl\"{u}cker coordinate $P$, we have that $\sigma_i(\varphi(P))=\varphi(\sigma_i(P))$. It is shown in \cite{Fra17} that for $|i-j|=1$, $\sigma_i \sigma_j \sigma_i(P) = \sigma_j \sigma_i \sigma_j(P)$. Therefore
\begin{align*}
\sigma_i \sigma_j \sigma_i(\varphi(P)) = \sigma_i \sigma_j (\varphi(\sigma_i(P))) = \varphi( \sigma_i \sigma_j \sigma_i(P)  )= \varphi(\sigma_j \sigma_i \sigma_j(P)) = \sigma_j \sigma_i \sigma_j(\varphi(P)).
\end{align*}

To show that $\sigma_i \sigma_j \sigma_i = \sigma_j \sigma_i \sigma_j$ on $\mathbb{C}_q[\Gr(k,n)]$, it suffices to shown that $\sigma_i \sigma_j \sigma_i(x_1 \cdots  x_r)=\sigma_j \sigma_i \sigma_j(x_1 \cdots  x_r)$, where $x_i$'s are quantum Pl\"{u}cker coordinates. This follows from the fact that $\sigma_i$, $\sigma_j$ are algebra automorphisms on $\mathbb{C}_q[\Gr(k,n)]$ and $\sigma_i \sigma_j \sigma_i(x_i)=\sigma_j \sigma_i \sigma_j(x_i)$ for quantum Pl\"{u}cker coordinates $x_i$.

The proof of $\sigma_i(\sigma_j(x)) = \sigma_j(\sigma_i(x))$ for every $x \in \mathbb{C}_q[\Gr(k,n)]$, $|i-j|>1$, is similar.


\begin{thebibliography}{99999}

\bibitem[ASS12]{ASS} I. Assem, R. Schiffler, and V. Shramchenko, \textit{Cluster automorphisms}, Proc. Lond. Math. Soc. (3) \textbf{104} (2012), no. 6, 1271--1302.

\bibitem[AG12]{AG12} J. Allman and J. E. Grabowski,
\textit{A quantum analogue of the dihedral action on Grassmannians},
J. Algebra 359 (2012), 49--68.

\bibitem[BFZ96]{BFZ96} A. Berenstein, S. Fomin, and A. Zelevinsky, \textit{Parametrizations of canonical bases and totally positive
matrices}, Adv. Math. \textbf{122} (1996), no. 1, 49--149.

\bibitem[BG02]{BG02} K. A. Brown and K. R. Goodearl, \textit{Lectures on algebraic quantum groups}, Advanced Courses in Mathematics, CRM Barcelona, Birkh\"{a}user Verlag, Basel, 2002.

\bibitem[BS15]{BS15}
T. Bridgeland and I. Smith,
\textit{Quadratic differentials as stability conditions}, 
Publ. Math. Inst. Hautes \'Etudes Sci. \textbf{121} (2015), 155--278.

\bibitem[BZ05]{BZ05} A. Berenstein and A. Zelevinsky,
\textit{Quantum cluster algebras}, Adv. Math. \textbf{195} (2005), no. 2, 405--455.

\bibitem[CS19]{CS19} W. Chang and R. Shiffler,
\textit{Cluster automorphisms and quasi-automorphisms}, Advances in Applied Mathematics, \textbf{110} (2019), 342–374. 

\bibitem[CZ16]{CZ16}
W. Chang and B. Zhu,
\textit{Cluster automorphism groups of cluster algebras with coefficients},
Sci. China Math. \textbf{59} (2016), no. 10, 1919--1936.

\bibitem[Fra16]{Fra16} C. Fraser,
\textit{Quasi-homomorphisms of cluster algebras}, Advances in Applied Mathematics, \textbf{81} (2016), 40--77.

\bibitem[Fra17]{Fra17} C. Fraser,
\textit{Braid group symmetries of Grassmannian cluster algebras}, Sel. Math. New Ser. \textbf{26} (2020), article number: 17. 

\bibitem[FWZ]{FWZ} S. Fomin, L. Williams, and A. Zelevinsky, \textit{Introduction to Cluster Algebras, Chapter 6}, preprint
(2020), arXiv:2008.09189 [math.CO].

\bibitem[FZ02]{FZ02}
S. Fomin, A. Zelevinsky,
\textit{Cluster algebras. {I}. {F}oundations.}
\newblock Journal of the American Mathematical Society, 2002, 15(2), 497-529.

\bibitem[GL11]{GL11} J. E. Grabowski and S. Launois,
\textit{Quantum cluster algebra structures on quantum Grassmannians and their quantum Schubert cells: the finite-type cases}, International Mathematics Research Notices, Volume \textbf{2011}, issue 10, 2011, 2230--2262.

\bibitem[GL14]{GL14} J. E. Grabowski and S. Launois,
\textit{Graded quantum cluster algebras and an application to quantum Grassmannians}, Proc. Lond. Math. Soc., \textbf{109} (2014), no. 3, 697--732.

\bibitem[JKS22]{JKS22} B. T. Jensen, A. King, X. Su,
\textit{Categorification and the quantum Grassmannian},
Advances in Mathematics,
Volume \textbf{406}, 2022, 108577.

\bibitem[KO19]{KO19}
Y. Kimura and H. Oya,
\textit{Twist automorphisms on quantum unipotent cells and dual canonical bases}, International Mathematics Research Notices, volume \textbf{2021}, issue 9, 2021, 6772--6847.

\bibitem[KQW22]{KQW22} Y. Kimura, F. Qin, and Q.l. Wei, \textit{Twist automorphisms and Poisson structures}, arXiv:2201.10284. 

\bibitem[Lau06]{Lau06} A. Lauve, \textit{Quantum- and quasi-Plücker coordinates}, Journal of Algebra, \textbf{296} (2006), 440--461.

\bibitem[Lau10]{Lau10} A. Lauve, \textit{Quasi-determinants and $q$-commuting minors}, Glasgow Math. J. \textbf{52} (2010), 663--675.
 
\bibitem[LL11]{LL11} S. Launois and T. H. Lenagan,
\textit{Twisting the quantum Grassmannian},
Proc. Amer. Math. Soc. 139 (2011), no. 1, 99--110.

\bibitem[LL23]{LL23} S. Launois and T. H. Lenagan, \textit{The automorphism group of the quantum grassmannian}, arXiv:2301.06865, 2023. 

\bibitem[LZ98]{LZ98}
B. Leclerc and A. Zelevinsky,
\textit{Quasicommuting families of quantum Pl\"{u}cker coordinates}, Kirillov's seminar on representation theory, 85--108, Amer. Math. Soc. Transl. Ser. 2, \textbf{181}, Adv. Math. Sci., \textbf{35}, Amer. Math. Soc., Providence, RI, 1998.

\bibitem[M16]{M16} G. Muller, \textit{Skein and cluster algebras of marked surfaces}, Quantum Topol. \textbf{7},  (2016), no. 3, 435--503. 

\bibitem[MS16]{MS16} B.R. Marsh and J.S. Scott, \textit{Twists of Pl\"{u}cker coordinates as dimer partition functions}, Commun. Math. Phys. \textbf{341}, 821--884 (2016).

\bibitem[OPS15]{OPS15} S. Oh, A. Postnikov, and D.E. Speyer, \textit{Weak separation and plabic graphs}, Proceedings of the London Mathematical Society \textbf{110}(3), 721--754 (2015).

\bibitem[Pos06]{Pos06} A. Postnikov, \textit{Total positivity, grassmannians, and networks}, 2006. \url{http://math.mit.edu/∼apost/papers/tpgrass}.

\bibitem[Sco05]{Sco05} J. Scott,
\textit{Quasi-commuting families of quantum minors},
Journal of Algebra, \textbf{290} (2005), 204--220.

\bibitem[Sco06]{Sco06} J. Scott, \textit{Grassmannians and cluster algebras}, Proc. London Math. Soc. (3) \textbf{92} (2006), no. 2, 345--380. 

\bibitem[TT91]{TT91} E. Taft and J. Towber, \textit{Quantum deformation of flag schemes and Grassmann schemes, I. A $q$-deformation of the shape-algebra for $GL(n)$}, J. Algebra \textbf{142}, no. 1, 1--36 (1991).

\end{thebibliography}
\end{document}